%% file: arxivDRAFT.tex
\documentclass{article}
\usepackage[utf8]{inputenc}
\usepackage{geometry}

\input{macros}

\title{Limit theorems for anisotropic functionals of stationary Gaussian fields with Gneiting covariance function}
\author{N.~Leonenko, L.~Maini, I.~Nourdin and F.~Pistolato\footnote{Corresponding author.\\MSC2020 subject classifications: Primary 60F05; secondary 60G15, 60G60, 60J55.}}
\date{March 6, 2026}

\begin{document}

\maketitle

\begin{abstract}
We study non-linear additive functionals of stationary Gaussian fields over anisotropically growing domains in $\mathbb{R}^d$, including spatiotemporal settings, and establish Gaussian and non-Gaussian limit theorems under non-separable covariance structures.
We characterize the regimes in which the normalized functionals converge either to a Gaussian distribution or to a $2$-domain Rosenblatt distribution, depending on precise long-range dependence conditions.
Our analysis covers covariance functions from the Gneiting class, which provides a canonical family of non-separable spatiotemporal models. A key structural result shows that such covariances are asymptotically separable in a precise cumulant sense, allowing us to identify explicitly the limiting distributions without imposing additional spectral assumptions.
These results extend existing spatiotemporal limit theorems beyond separable and short-memory frameworks and provide a unified description of anisotropic long-range dependence phenomena.
\end{abstract}

\section{Introduction}

In recent years, increasing attention has been devoted to spatiotemporal Gaussian models, motivated by applications in epidemiology, environmental sciences, economics, and social sciences; see \cite{systematic} for a recent review of spatiotemporal random models over the past 5 years. From a probabilistic perspective, such models correspond to Gaussian fields indexed by product spaces $\R^{d_1}\times\R^{d_2}$, where $x_1\in\R^{d_1}$ typically represents a spatial variable and $x_2\in\R^{d_2}$ a temporal one. While the classical case corresponds to $d_1=2,3$ and $d_2=1$, higher-dimensional index sets naturally arise when modeling more complex dependence structures. Related work in this direction includes \cite{porcu1motivation}, which studies multivariate Gaussian random fields on generalized product spaces.

In this paper, we investigate nonlinear additive functionals of stationary Gaussian fields under \emph{anisotropic growth of the observation domain}. More precisely, let $B=(B_x)_{x\in\R^d}$, $d\ge2$, be a continuous, stationary, centered Gaussian field with unit variance and covariance function
\begin{equation}
\label{eq:cov}
    C(x) := \Cov(B_x,B_0), 
    \qquad x\in\R^d.
\end{equation}
Let $\varphi:\R\to\R$ be a non-constant function such that $\E[\varphi(B_0)^2]<\infty$. Then, $\varphi$ admits the Hermite expansion
\begin{equation}
\label{eq: deco phi}
    \varphi \eqLtwo \sum_{q\in\N} a_q H_q, \qquad a_q := \frac1{q!}\E[H_q(N)\varphi(N)],
\qquad N\sim\mathcal N(0,1),
\end{equation}
where $H_q$ denotes the $q$-th Hermite polynomial.
The \emph{Hermite rank} of $\varphi$ is the smallest $q\ge1$ such that $a_q\neq0$.

Let $d_1,d_2\ge1$ with $d_1+d_2=d$, and let $D_i\subset\R^{d_i}$ $i=1,2$, be convex bodies, that is, convex compact sets with non-empty interior. We consider the additive functional
\begin{equation}
\label{eq:Yt}
    Y(t) 
    = \int_{t_1(t)D_1 \times t_2(t)D_2} 
      \varphi(B_x)\, dx,
    \qquad t\ge0,
\end{equation}
where $t_i(t) \to \infty$ as $t\to\infty$, for $i=1,2$, and \begin{equation}
\label{domain}
    t_1(t)D_1 \times t_2(t)D_2 : = \{(t_1(t)x_1,t_2(t)x_2)\in\R^{d_1}\times \R^{d_2}: (x_1,x_2)\in D_1\times D_2\}.
\end{equation}
The functional \eqref{eq:Yt} is well-defined since $B$ is continuous and $\varphi$ is square-integrable. Provided that $\Var(Y(t))> 0$ for $t$ large enough, we study the limit in distribution as $t\to\infty$ of \begin{equation}
\label{eq:Yttilde}
    \widetilde Y(t) : = \frac{Y(t) - \E[Y(t)]}{\sqrt{\Var(Y(t))}}.
\end{equation}

When $t_1\equiv t_2$, the problem falls within the classical theory of nonlinear functionals of Gaussian fields. In the discrete setting ($\Z^d$), central and non-central limit theorems were established in the seminal works \cite{BM83,DM79,T79}. Malliavin-calculus-based approaches are presented in \cite[Chapter~7]{bluebook}, and further extensions in the continuous framework appear in \cite{Har02continuousFuncBM,CNN20tightnessNonStat,MN22,Maini2024}. However, these results do not address \emph{anisotropic growth}, namely the case where different components of the domain diverge at different rates.

Anisotropic regimes have recently been investigated in \cite{RST,PR16,pdom24}. In particular, \cite{pdom24} provides a characterization of the limiting distribution of \eqref{eq:Yttilde} under the crucial assumption that the covariance is \emph{separable}, i.e.
\begin{equation}
\label{eq:separable}
    C(x_1,x_2)
    = C_1(x_1)C_2(x_2), \qquad \forall x=(x_1,x_2)\in \R^{d_1}\times \R^{d_2}.
\end{equation}
In the same work, it is shown that models with non-separable covariance structures may exhibit significantly different limiting behavior, and different distributional limits can be obtained only changing the rate at which $t_1$, and $t_2$ diverge.
While separable models are mathematically convenient and widely used in applications (see, e.g., \cite{DDK25testingseparability}), they impose strong structural restrictions and fail to capture more intricate space–time interactions.

\subsection*{Main contribution}
The purpose of the present work is to investigate the asymptotic behavior of \eqref{eq:Yttilde} in the \emph{fully non-separable} setting where the covariance function belongs to the Gneiting class \cite{Gne02}, namely
\begin{equation*}
C(x_1,x_2)
=
C_2(x_2)\,
C_1\!\left(x_1\, C_2(x_2)^{1/d_1}\right),
\end{equation*}
under suitable regularity assumptions on $C_1$ and $C_2$ ensuring positive definiteness (see \Cref{ass:gneiting}). Our contribution is threefold.
\begin{itemize}
    \item[(i)] We establish Gaussian and non-Gaussian limit theorems for anisotropic nonlinear functionals under fully non-separable Gneiting-type covariances.
    \item[(ii)] We prove that Gneiting covariances are \emph{asymptotically separable} in a precise sense.
    \item[(iii)] We identify explicitly when $2$-domain Rosenblatt limits arise, without imposing additional spectral assumptions.
\end{itemize}

Although they are fully non-separable, we prove that Gneiting-type covariances become asymptotically separable in the sense of \cref{def:asympsep}. Therefore, the limiting behavior coincides with that of an appropriate separable model at large scales. This phenomenon provides an explanation for the emergence of Rosenblatt-type limits in anisotropic regimes.

We emphasize that, although the Gneiting construction imposes structural constraints, it remains one of the very few general and constructive criteria for building valid non-separable space-time covariance functions on $\R^{d_1}\times\R^{d_2}$. See also \cite{MR1731494,MR3573273,MR4214159,MR2156840}. From this perspective, the Gneiting class provides a canonical framework for investigating genuinely non-separable phenomena. To the best of our knowledge, this is the first systematic study of nonlinear space–time functionals in a fully non-separable setting, on Euclidean product spaces. 

In the non-central regime, we focus on functionals of Hermite rank $R=2$. This corresponds to the first genuinely non-Gaussian case and already encompasses a large class of relevant statistical and geometric functionals, including quadratic variations, power variations, and Lipschitz–Killing curvatures of random fields. 

Finally, our method applies more generally to Gaussian fields with asymptotically separable covariance functions (see Section~\ref{sec:asympsep}) and is not restricted to the classical spatiotemporal case $d_2=1$, but holds for arbitrary $d_1,d_2\ge1$.

\subsection*{Related results} 

During the final stage of this work, we became aware of the preprint \cite{ruizmedina2026noncentrallimitresultsspatiotemporal}, where non-central limit theorems for spatiotemporal functionals of Gaussian fields on $\mathbb{R} \times \mathbb{R}^d$ are established under the assumption that the covariance is asymptotically separable and satisfies additional spectral conditions. In contrast, our starting point is a fully non-separable covariance structure. We prove that asymptotic separability emerges for Gneiting type covariances, and we identify explicitly the $2$-domain Rosenblatt limits, without imposing further spectral assumptions. 

We also compare our results with \cite{LRM23,LRM25}, where the authors study spatiotemporal functionals in the case $d_2=1$ and $D_2=[0,1]$, assuming $t_2(t)=t$. They establish reduction principles and derive central limit theorems after suitable normalization. However, their results are restricted to Hermite rank $R=1$, and do not address higher-order chaotic components nor the fully non-separable regime considered here. Moreover, their conclusions apply only partially to the Gneiting class, essentially when the spatial scaling satisfies $t_1(t)=t^\gamma$ for suitable $\gamma>0$. Our results cover arbitrary anisotropic growth rates and provide a complete characterization of Gaussian and non-Gaussian limits in the case $R=2$.

Related investigations of sphere-cross-time Gaussian fields, where the spatial variable belongs to $S^2$ and time evolves continuously, have been recently conducted in \cite{smallFluctSphereTime,MRV21MotivNonUni,fluctLevelSphereTime,caponera2025fractionalcointegrationgeometricfunctionals}. In \cite{MRV21MotivNonUni,fluctLevelSphereTime}, the authors analyze long-time fluctuations of geometric functionals such as excursion areas and level curve lengths for isotropic-in-space, stationary-in-time Gaussian fields on $S^2\times\R^+$. Depending on the temporal memory and on the dominant Wiener chaos component, both Gaussian and non-Gaussian (including Rosenblatt-type) limits arise. In contrast to our setting, these results concern long-time asymptotics over a fixed spatial manifold ($S^2$), but address general fully non-separable covariance structures of spatiotemporal fields on the sphere. In \cite{caponera2025fractionalcointegrationgeometricfunctionals}, the authors show that the aforementioned functionals exhibit fractional cointegration, with applications to statistical estimation (see also \cite{MR4244190}). Sojourn functionals of time-dependent random fields have been investigated outside of the Gaussian framework, too, see e.g. \cite{CRRm25}. 

\subsection*{Some notation} 

In the following, we write $f(t) \asymp g(t)$ as $t\to\infty$ to denote that there exist constants $A, B > 0$ such that $A\, g(t) \le f(t) \le B\, g(t)$ for $t$ large enough. If only the right-hand side inequality is valid, we write $f(t) \lesssim g(t)$. We also write $f(t) \sim g(t)$ as $t\to\infty$ if $\lim_{t\to\infty} f(t)g(t)^{-1} = 1$. Given a probability space $(\Omega,\mathcal F,\mathbb P)$, and a random variable $X$ such that $\Var (X) >0$, we write $\widetilde {X}$ to denote its standardization: $\widetilde {X} = \frac{X-\E[X]}{\sqrt{\Var X}}$.

\subsection{Setting and main results}

In the framework of Equations \eqref{eq:cov}-\eqref{eq:Yttilde} presented above, we formulate our results under the following assumptions. 

\begin{assumption}[Gneiting type]
\label{ass:gneiting}
We assume that $C$ belongs to the Gneiting class\footnote{In \cite{Gne02}, the author proved that the previous assumptions on $C_1$ and $C_2$ ensure positive-definiteness of $C$. We remark that the above result has been recently generalised in \cite{extending_gneiting}.}, namely
\begin{equation}
\label{eq:Cgneiting} 
C(x) = C_2(x_2) \, C_1 \left( x_1\, C_2(x_2)^{1/d_1} \right), \qquad \forall x = (x_1,x_2)\in \R^{d_1}\times \R^{d_2},
    \end{equation}
    where $C_1:\R^{d_1}\to\R$ and $C_2:\R^{d_2}\to\R$ are such that \begin{itemize}
    \item both $C_1$ and $C_2$ are radial functions: there exist $c_i:\R^+_0\to\R$, $i=1,2$, such that \begin{equation*}
        C_i(x_i) = c_i(\|x_i\|), \qquad \forall x_i\in\R^{d_i}, 
    \end{equation*}
    \item $c_1$ is completely monotone\footnote{A function $f : \R^+_0 \to \mathbb{R}$ is said to be \emph{completely monotone} if it is of class $C^\infty$ and satisfies \begin{equation*}
        (-1)^n f^{(n)}(x) \ge 0 \quad \text{for all } x>0 \text{ and all } n \in \mathbb{N}_0 .
    \end{equation*}
    Complete monotonicity of $c_1$ can in fact be relaxed, see \cite{extending_gneiting}.} (in particular, $C_1\ge0$),
    \item $C_2>0$, and $\frac1{c_2}$ has a completely monotone derivative.
\end{itemize} 
\end{assumption}

On top of the previous assumption, we will assume that either $C_1$ or $C_2$, or both, are regularly varying according to the following definition.

\begin{definition}[Regularly varying function]
\label{ass:RV}
    A function $G:\R^d \to \R$ is \emph{regularly varying} at infinity with index $-\rho<0$ if 
    \begin{equation*}
        G(|x|) \sim |x|^{-\rho} \, L(|x|), \qquad \text{as } |x|\to\infty,
    \end{equation*}
    where $L:\R\to\R$ is a \emph{slowly varying} function at infinity, that is, \begin{equation}
    \label{eq:SV}
        \lim_{|x| \to \infty}  \frac{L(\alpha x)}{L(x)} = 1, \qquad \forall \alpha > 0.
    \end{equation}
    If, in addition, $\lim_{x \to \infty} L(x) = M \neq 0$, then we say that $G$ is \emph{asymptotically power-law} with exponent $-\rho$. In the following, we omit to explicit mention ``at infinity''.
\end{definition}

\begin{remark}
    In the following, we will assume that either $C_1$, $C_2$, or both are power laws with negative exponents. We remark that the nature of the results doesn't change when considering proper fregularly varying functions.
\end{remark}

Our next assumption entails the rate at which the windows $t_1(t)D_1$ and $t_2(t)D_2$ are growing as $t\to\infty$, with an interaction with the Gneiting structure of $C$. It ensures that the effective spatial rescaling induced by the Gneiting structure diverges sufficiently fast.

\begin{assumption}[Rate]
\label{ass:KEY} 
Assume that $C_2$ is radial, with associated radial covariance function $c_2$, and that \begin{equation*}
    \lim_{t\to+\infty} t_1(t) c_2(t_2(t))^{1/d_1} = +\infty\,. 
\end{equation*}
If $C_2$ is asymptotically power-law with exponent $-\rho_2<0$, the previous assumption is satisfied if \begin{equation*}
    \lim_{t\to\infty} t_1(t)t_2(t)^{-\rho_2/d_1} = \infty.
\end{equation*}
\end{assumption}

We are now in a position to state our main results: two central limit theorems 
for functionals of generic Hermite rank $R\ge 2$, 
\Cref{thm:ltBREUERMAJOR} and \Cref{thm:ltBALTO}, 
and a non-central one for functionals of Hermite rank $R=2$, 
\Cref{thm:ltROSENBLATT}. 

\begin{remark}[The case $R=1$]
When the Hermite rank of $\varphi$ is $R=1$, the functional $Y(t)$ is asymptotically fully correlated to its first chaotic component, which is linear in the underlying Gaussian field, see indeed \cref{prop:reduction}, and also \cref{rem:R1variance}. Therefore, after normalization, 
$\widetilde Y(t)$ is asymptotically Gaussian.
\end{remark}

\begin{theorem}
\label{thm:ltBREUERMAJOR}
    Let us consider $Y(t)$ as in \eqref{eq:Yt}. Suppose Assumptions \ref{ass:gneiting} and \ref{ass:KEY} hold. Let $R\ge2$ be the Hermite rank of $\varphi$.
    If $C\in L^R(\mathbb R^{d})$, then, as $t\to\infty$,  
    \begin{equation*}
        \widetilde Y(t) \convlaw \mathcal N(0,1).
    \end{equation*}
    Moreover, $\Var(Y(t))\sim \ell \ t_1(t)^{d_1}t_2(t)^{d_2}$, where \begin{equation*}
        \ell = \vol(D_1)\vol(D_2) \, \sum_{q\ge R} q! \, a_q^2 \, \|C_1\|^q_{L^q(\R^{d_1})}\|C_2\|^{q-1}_{L^{q-1}(\R^{d_2})}.
    \end{equation*} 
    If $C_1\in L^{R}(\mathbb R^{d_1})$, and $C_2\notin L^{R-1}(\mathbb R^{d_2})$, but $C_2$ is slowly or regularly varying, then, as $t\to\infty$, 
    \begin{equation*}
        \widetilde Y(t) \convlaw \mathcal N(0,1).
    \end{equation*}
    Moreover, $\Var(Y(t)) \sim \ell^\prime t_1(t)^{d_1} t_2(t)^{2d_2} \, \int_{t_2(t)D_2} C_2(z_2)^R \, dz_2$, for some positive constant $\ell^\prime$.
\end{theorem}

\begin{remark}
The previous result clarifies the role of \Cref{ass:KEY}. 
In particular, it shows the perhaps counter-intuitive fact that short-range dependence in the first block alone (i.e. $C_1\in L^R(\mathbb R^{d_1})$) is sufficient to determine the asymptotic behavior of the entire two-domain functional, even when the second block exhibits long-range dependence.
\end{remark}

\begin{theorem} 
\label{thm:ltBALTO}
    Let us consider $Y(t)$ as in \eqref{eq:Yt}. Let $R\ge2$ be the Hermite rank of $\varphi$. Suppose Assumptions \ref{ass:gneiting} and \ref{ass:KEY} hold. 
    If $C_1\notin L^R(\R^{d_1})$, but it is asymptotically a power-law with exponent $-\rho_1<0$ such that $R\rho_1<d_1$, and if $C_2\in L^{R(1-\rho_1/d_1)}(\R^{d_2})$, then, as $t\to\infty$, \begin{equation*}
        \widetilde Y(t) \convlaw \mathcal N(0,1).
    \end{equation*}
    Moreover, $\Var(Y(t))\sim \ell^{\prime\prime} \, t_2(t)^{d_2} \, t_1(t)^{2d_1-R\rho_1}$, for some positive constant $\ell^{\prime\prime}$.
\end{theorem}

\begin{theorem} 
\label{thm:ltROSENBLATT}
    Let us consider $Y(t)$ as in \eqref{eq:Yt}. Let $R$ be the Hermite rank of $\varphi$. Suppose Assumptions \ref{ass:gneiting} and \ref{ass:KEY} hold, and that $R=2$. 
    If $C_1\notin L^2(\mathbb R^{d_1})$ and $C_2\notin L^1(\mathbb R^{d_2})$, but they both are asymptotically power-law with exponents $-\rho_1<0$ and $-\rho_2<0$, respectively, such that $2\rho_1<d_1$ and $\rho_2<\frac{d_1d_2}{2(d_1-\rho_1)}$, then, as $t\to\infty$,
    \begin{equation*}
    \widetilde Y(t) \convlaw \textrm{sgn}(a_2) \,  H_{2,D_1\times D_2}^{\rho_1,\rho_2\left(1-\frac{\rho_1}{d_1}\right)},
\end{equation*}
where $a_2$ is as in \eqref{eq: deco phi}, and $ H_{2,D_1\times D_2}^{\rho_1,\rho_2\left(1-\frac{\rho_1}{d_1}\right)}$ is a $2$-domain Rosenblatt r.v. as defined in \cref{sec:rosenblatt}.
Moreover, $\Var(Y(t)) \sim \ell^{\prime\prime\prime} \, t_1(t)^{2d_1-R\rho_1} \, t_2(t)^{2d_2-R\rho_2(1-\rho_1/d_1)}$, for some positive constant $\ell^{\prime\prime\prime}$.
\end{theorem}

\begin{remark}[Criticality] In the critical cases when $2\rho_1 = d_1$ or $\rho_2 = \frac{d_1 d_2}{2(d_1 - \rho_1)}$, the variance is expected to exhibit logarithmic corrections. The present analysis can be adapted to cover these critical cases, as well as the case $R\rho_1=d_1$ in \cref{thm:ltBALTO}; however, this would require lengthy but straightforward refinements of the asymptotic expansions for regularly varying integrals, which we omit for brevity.
\end{remark}

\begin{remark}
Figure~\ref{fig:lt} provides a comprehensive view of the results stated in Theorems \ref{thm:ltBREUERMAJOR}, \ref{thm:ltBALTO}, and \ref{thm:ltROSENBLATT}, under the assumptions that $R=2$, and $C_1$ and $C_2$ are asymptotically power-laws with exponents $-\rho_1<0$ and $-\rho_2<0$, respectively.
\end{remark}

\begin{figure}
\centering
\begin{tikzpicture}[>=stealth, scale=1]
\draw[->] (-0.5,0) -- (8,0) node[right] {$\rho_1$};
\draw[->] (0,-0.5) -- (0,5) node[left] {$\rho_2$};

\draw[very thick,cyan] (4,0) -- (4,4);
\draw[very thick,teal] (4,4) -- (4,4.5);
\draw[dashed] (4,4.5) -- (4,5);

\draw[dotted,gray] (0,4) -- (4,4);
\draw[very thick,magenta] (0,2) to [bend right=25] (4,4);
\draw[very thick,red] (4,4) -- (7.5,4);
\draw[dashed] (7.5,4) -- (8,4);

\node[left] at (0,2) {$\frac{d_2}{2}$};
\node[left] at (0,4) {${d_2}$};
\node[below] at (4,0) {$\frac{d_1}{2}$};


\node[align=center] at (2,4) {$t_1^{2d_1-{2}\rho_1}t_2^{d_2}$\\Gaussian};

\node[align=center] at (6.1,4.75) {$t_1^{d_1}t_2^{d_2}$\\Gaussian};

\node[align=center] at (6.2,2) {$t_1^{d_1}t_2^{2d_2-{}\rho_2}$\\Gaussian};

\fill[gray!20,opacity=0.5]
    (0,0)
    -- (0,2)
    to [bend right=25] (4,4)
    -- (4,0)
    -- cycle;
\node[align=center] at (2,1) {$t_1^{2d_1-{2}\rho_1}t_2^{2d_2-{2}\rho_2(1-\frac{\rho_1}{d_1})}$\\ $2$-domain Rosenblatt};
\end{tikzpicture}

\caption{\small Four non-critical regimes of the variance of $Y_{t_1,t_2}$ (up to constants) and limiting distribution of $\widetilde Y_{t_1,t_2}$, when the Hermite rank of $\varphi$ is $R=2$, and both $C_1$ and $C_2$ are regularly varying with parameters $-\rho_1$ and $-\rho_2$, respectively, with constant slowly varying components. In magenta, the line $\rho_2 = \frac{d_1d_2}{2(d_1-\rho_1)}$.  
The limiting distribution is everywhere Gaussian, except for $(\rho_1,\rho_2)$ in the gray region delimited by the magenta and cyan curves, where the limiting distribution is a random variable in the second chaos.
}
\label{fig:lt}
\end{figure}
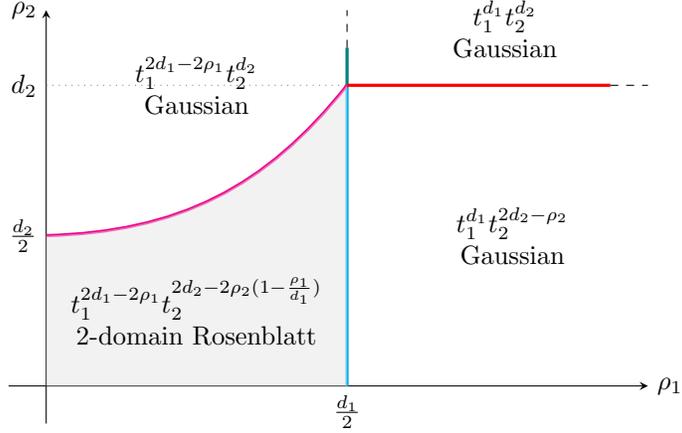

\begin{remark}
We emphasize that \cref{thm:ltBREUERMAJOR} (in the case $C_2\notin L^{R-1}(\R^{d_2})$) and \cref{thm:ltBALTO} provide two instances of a broader class of functionals exhibiting Gaussian fluctuations in the presence of long-range dependence, outside the critical regime. In both situations, the normalized functional $\widetilde Y(t)$ satisfies a Central Limit Theorem, while the variance grows superlinearly with respect to the observation volume, namely
\begin{equation*}
    \Var(Y(t)) \gg \vol\!\big(t_1(t)D_1\times t_2(t)D_2\big),
    \qquad t\to\infty.
\end{equation*}
To the best of our knowledge, examples of this phenomenon were first identified in \cite{MN22} and subsequently extended in \cite{pdom24}.
\end{remark}

\subsubsection{Overview of the proof}
\textbf{Notation.} In the following, we will write $t_1$ and $t_2$ omitting the dependence on $t$. Recall that $t_1,t_2\to \infty$ as $t\to\infty$.

\medskip

We refer to \cref{sec:prelim} for all needed preliminaries, and \cref{subsec:proofLT} for a complete proof.
In this overview, we suppose for brevity that $\f = H_2$, the $2$nd Hermite polynomial. Since $\widetilde Y(t)$ belongs to the second Wiener chaos, 
its distribution is moment-determined (see \cref{fredholm does the job}). Because moments and cumulants uniquely determine each other, it suffices to analyze the cumulants of $\widetilde Y(t)$, denoted by $\kappa_k(t)$ for $k\ge 1$. As $\widetilde Y(t)$ is centered and standardized, we have $\kappa_1(t)=0$ and $\kappa_2(t)=1$ for every $t>0$. For higher $k\ge 2$, we have that \begin{equation}
\label{eq:kappaktilde}
    \kappa_k(t) = 2^{k-1}(k-1)! \, c_k(t),
\end{equation}
where \begin{equation*}
    c_k(t) := \frac{\int_{(t_1D_1 \times t_2D_2)^k} C(x^{(1)}-x^{(2)})C(x^{(2)}-x^{(3)}) \cdots C(x^{(k)}-x^{(1)}) dx^{(1)}  \ldots dx^{(k)}}{\left(\int_{(t_1D_1 \times t_2D_2)^2} C^2(x^{(1)}-x^{(2)}) dx^{(1)}dx^{(2)}\right)^{k/2}},
\end{equation*}
and $C$ is the covariance function of the underlying Gaussian field $B$. Remark that, by H\"older inequality, $\{c_k(t)\}_t$ is uniformly bounded.

In \cref{thm:ltBREUERMAJOR,thm:ltBALTO}, 
we prove that $c_4(t)\to 0$ as $t\to\infty$. The Fourth Moment Theorem (see \cref{thm:4th}) then yields asymptotic normality. In particular, this implies $c_k(t)\to 0$ for all $k\ge 3$, since $\kappa_k(t)\to \kappa_k(N)=0$ for every $k\ge 3$.

In \cref{thm:ltROSENBLATT}, under additional long-range dependence assumptions on $C_1$ and $C_2$ and regular variation, we explicitly compute the limit \begin{equation}
\label{eq:limit}
    \lim_{t \to\infty} c_k(t), \qquad \forall k\ge 2,
\end{equation}
and show that the limiting cumulants coincide with those of a normalized 
$2$-domain Rosenblatt random variable, introduced in \cite{pdom24}. 
This random variable belongs to the second Wiener chaos and is non-Gaussian. Since distributions in the second Wiener chaos are moment-determined, convergence of cumulants implies convergence in law.
To compute the limit \eqref{eq:limit}, we exploit the properties of regularly varying functions, which may be controlled by suitable power-laws by means of Potter bounds, see \cref{prop:potter}. 

\begin{remark}
   The coefficients $c_k(t)$ appear already in the seminal paper \cite[Remark 2]{cltNuaPec} as traces of $k$th power of the Hilbert–Schmidt operator associated to the covariance kernel of the underlying Gaussian field $B$. 
\end{remark}

\subsubsection{Further remarks and example}

\begin{remark}[On an alternative characterization via Fredholm determinants]
We also note that moment determinacy in the second Wiener chaos can be established through an alternative approach based on Fredholm determinants. Indeed, random variables in the second chaos admit a representation in terms of trace class Hilbert--Schmidt operators, and their characteristic functions can be expressed via Fredholm determinants of the associated covariance operators. This provides a method entirely equivalent to the cumulant analysis employed here. This operator-theoretic approach is pursued, for instance, in \cite{ruizmedina2026noncentrallimitresultsspatiotemporal}.
\end{remark}

\begin{remark}\label{rem:whyR2}
In the non-central case, although we assume that $R=2$, we can cover many relevant examples. This is the case of $p$-power variations of stationary Gaussian processes, for any integer $p\ge 2$; this class includes the quadratic variation, see \cite{BNCP09}. Many examples of geometric functionals of Gaussian fields have Hermite rank either $1$ or $2$, too: volumes of excursion sets of Euclidean and spherical random fields are closely related to the study of their empirical measures, see \cite{LRM23,MariWig11}. 
It is worth noting that, in statistical applications, ranks higher than $1$ are considered \emph{unstable}, namely a small perturbation of the underlying field can lead to an abrupt change in the Hermite rank of the functional, see the discussion in \cite{BT18}. To conclude, we mention the recent paper \cite{Arp26} that unveils that \emph{typically} sinh-arcsinh transformations of Gaussian processes have Hermite rank $1$ or $2$. 
Chaotic decompositions for other geometric functionals of Gaussian fields have also been established on the torus, see, e.g., \cite{VidottoLKC22,CMRLKC23}; on the $n$-sphere and on general manifolds, see, e.g., \cite{MR4020703,stecconi2025newchaosdecompositiongaussian,pistolato2025levelareaspinrandom}. Recent contributions extend the present framework to non-local functionals as well, see \cite{mcauley2025limittheoremsnumbersign}, and to non-Gaussian fields, see \cite{gass2025universalcancellationsuniformrandom}. Recent contributions extend the present framework to non-local functionals as well, see \cite{mcauley2025limittheoremsnumbersign}, and to non-Gaussian fields, see \cite{gass2025universalcancellationsuniformrandom}.
\end{remark}

\begin{remark}[On the spectral density]
\label{rem: no spectral reg var}
Note that, in \cite[(1.7)-(1.8)]{pdom24}, both regular variation of the covariance function at $\infty$ and absolute continuity and regular variation at $0$ for the associated spectral measure are assumed. In the non-central \cref{thm:ltROSENBLATT}, the second spectral assumption is not needed, because the Hermite rank is $R=2$. 

Indeed, when $R=2$ one can reason as follows: given a regularly varying covariance function $C$ with index $-\rho$, we can take an auxiliary Cauchy covariance function $C'(\|x\|)=(1+\|x\|^2)^{-\rho/2}$ with spectral density  regularly varying in $0$ (see e.g. \cite[Example 3]{LO13}). For $R=2$, the limiting distribution is characterized by the cumulants, therefore by the coefficients $c_k$. The limiting coefficients $c_k(\infty,D)=\lim_{t\rightarrow\infty}c_k(t,D)$ (see \eqref{eq:ck}) are equal if we consider $C$ or $C'$, as we shall prove in \cref{lem:cactus}, with no further assumption on the existence of the spectral density.
\end{remark}

\begin{example}
\label{example}

We can apply the previous results to the study of the first Minkowski functional of the modulus of a Gaussian field with Gneiting covariance function, that is, the temporal average of its excursion volume over a time-evolving window; see also \cite{LRM23,LRM25}.

Consider a centered and stationary Gaussian field $B=(B_x)_{x\in\R^d}$ with zero mean. Suppose that its covariance function $C$ is of Gneiting type, with factors $C_1$ and $C_2$ as in \Cref{ass:gneiting}. Assume that both $C_1$ and $C_2$ are asymptotically power laws with exponents $-\rho_1<0$, and $-\rho_2<0$, respectively. Consider $\gamma>0$, and $K\subset\R^d$ a convex body, and the functional \begin{equation*}
    M_T := \int_0^T \int_{T^\gamma K} \mathbbm{1}(|B_x|\ge u) \, dx , \qquad T>0,
\end{equation*}
for a given level $u>0$. It is well known that $M_T$ admits the following chaotic decomposition: \begin{equation*}
    M_T = 2(1-\Phi(u)) T^{1+\gamma d} \vol(K) + \sum_{\text{even } q\ge 2} \frac{2 H_{q-1}(u)\phi(u)}{q!} \int_0^T \int_{T^\gamma K} H_q(B_x) \, dx;
\end{equation*}
see e.g. \cite{LRM25}. As a consequence of \cref{thm:ltBREUERMAJOR,thm:ltBALTO,thm:ltROSENBLATT}, we can prove the following result.

\begin{corollary} Let $\gamma\in (0,1/\rho_2)$. 
Then, if $u>0$, we have \begin{equation*}
    \Var(M_T) \asymp \begin{cases}
        T^{\gamma d + 1} & \text{if } C\in L^2(\R^{d+1}) ,\\
        T^{\gamma d + 2 - \rho_2} & \text{if } C_1\in L^2(\R^{d}), \text{ and } C_2\notin L^{1}(\R): \rho_2< 1,\\
        T^{2\gamma(d-\rho_1)+1} & \text{if } C_1\notin L^2(\R^{d}): \rho_1<d/2, \text{ and } C_2\in L^{2(1-\rho_1/d)}(\R),\\
        T^{2\gamma(d-\rho_1) + 2(1 - \rho_2(1-\rho_1/d))} &   \text{if } C_1\notin L^2(\R^{d}): \rho_1<d/2, \text{ and } C_2\notin L^{2(1-\rho_1/d)}(\R): \\ 
        &  \hfill \rho_2< \frac{d}{2(d-\rho_1)}. 
    \end{cases}
\end{equation*} 
In the first three cases, \begin{equation*}
    \widetilde M_T := \frac{M_T - 2(1-\Phi(u)) T^{1+\gamma d} \vol(K)}{\sqrt{\Var(M_T)}} \convlaw N ,
\end{equation*}
where $N$ is a standard Gaussian random variable. In the fourth case, $\widetilde M_T$ converges to a $2$-domain Rosenblatt $ H^{\rho_1,\rho_2(1-\rho_1/d)}_{2,K\times [0,1]}$, as in \cref{def:2domainRosenblatt}.
\end{corollary}
\end{example}

\subsection{Asymptotic separability}\label{sec:asympsep}

Consider a positive-definite function $f:\R^d \to \R$, and a convex body $D\subset\R^{d}$. For all $k\in\N$, $k\ge 2$, we define  \begin{equation}\label{eq:ck}
    c_k(D;f) := \frac{\int_{D^k} f(x_1-x_2)f(x_2-x_3) \cdots f(x_k-x_1) \, dx_1  \ldots dx_k}{\left(\int_{D^2} f(x_1-x_2)^2 \, dx_1 dx_2 \right)^{k/2}}.
\end{equation}
We remark that $c_2(D;f) = 1$, and, if $f\le 1$, by H\"older inequality, $c_k(D;f)$ is bounded, for all $k\ge 3$. 

\begin{definition}[Asymptotic separability]\label{def:asympsep}
    We say that a non-negative definite function $f:\R^d \to \R$ is \emph{asymptotically separable} onto $\R^{d_1}\times \R^{d_2}$, such that $d_1+d_2=d$, if there exist two non-negative definite functions $f_i:\R^{d_i}\to\R$, $i=1,2$, such that, for any convex bodies $D_i\subset\R^{d_i}$, $i=1,2$, and $k\ge 2$, we have that \begin{equation*}
        \lim_{t_1,t_2\to\infty} \left| c_k(t_1D_1\times t_2D_2;f) - \prod_{i=1,2} c_k(t_iD_i;f_i) \right| = 0.
    \end{equation*}
\end{definition}
Clearly, a separable covariance function, as defined in \eqref{eq:separable}, is asymptotically separable in the sense of the previous definition.

\begin{theorem}\label{thm:equivalence}
    Let $B=(B_x)_{x\in\R^d}$ be a real-valued, continuous, centered and unit-variance, stationary Gaussian field on $\R^d$. Let $C$ denotes its covariance function. For $i=1,2$, consider a convex body $D_i\subset \R^{d_i}$, and define
    \begin{equation*}
        Y_{t_1,t_2} = \int_{t_1D_1 \times t_2D_2} H_2(B_x) dx, \qquad t_1,t_2>0.
    \end{equation*}
    Suppose that $C$ is asymptotically separable onto $\R^{d_1}\times \R^{d_2}$. Then, there exists a real-valued, continuous, centered and unit-variance, stationary Gaussian field $B^{\textrm{Sep}} =(B_x^{\textrm{Sep}})_{x\in\R^d}$ with separable covariance function $C^{\textrm{Sep}}$, such that, given 
    \begin{equation*}
         Y_{t_1,t_2}^{\textrm{Sep}} :=  \int_{t_1 D_1 \times t_2 D_2}   H_2 (B_x^{\textrm{Sep}}) \, dx ,
    \end{equation*}
    the following are equivalent: \begin{itemize}
        \item $\widetilde Y_{t_1,t_2}^{\textrm{Sep}}\convlaw Z$, as $t_1,t_2 \to \infty$,
        \item $\widetilde Y_{t_1,t_2}\convlaw Z$, as $t_1,t_2 \to \infty$,
    \end{itemize}
    where $Z$ is the same random variable. In other words, $\widetilde Y_{t_1,t_2}$ and $\widetilde Y_{t_1,t_2}^{\textrm{Sep}}$ share the same limit in distribution, if there is one.
\end{theorem}

\begin{proof}
Since $\f=H_2$, the functional $Y_{t_1,t_2}$ belongs to the second chaos. Therefore, its distribution is determined by its moments or, equivalently, by the cumulants $\kappa_k$, for $k\ge 2$; moreover, for all $k\ge 2$, the following equality holds: \begin{equation*}
        \kappa_k (\widetilde Y_{t_1,t_2}) = 2^{k-1}(k-1)! \,2^{-k/2} \, c_k(t_1D_1\times t_2D_2;C),
    \end{equation*}
where $\widetilde Y$ denotes the standardization of $Y$; see \cite[Proposition 2.1.7.13]{bluebook}, or \cref{fredholm does the job}, below. If $C$ is asymptotically separable, we know that, as $t_1,t_2\to\infty$, \begin{equation*}
    \kappa_k (\widetilde Y_{t_1,t_2}) \sim 2^{k-1} (k-1)! \, 2^{-k/2} \, \prod_{i=1,2} c_k(t_iD_i;C_i) = \kappa_k (\widetilde Y^{\mathrm{Sep}}_{t_1,t_2}), \quad k\ge 2,
\end{equation*}
for some random variable $Y^{\mathrm{Sep}}$ subordinated to a Gaussian field with separable covariance function. Since a double integral converges in law to a r.v. $Z$ if and only if all its cumulants converge to those of $Z$, see \cite{nourdinPoly12}, they share the same limiting distribution. 
\end{proof}

\begin{remark}[Asymptotic separability on $p\ge2$ blocks]
    If $p\ge 3$, we say that $f:\R^d \to \R$ is \emph{asymptotically separable onto $p$ blocks} $\R^{d_1}\times \ldots \times \R^{d_p}$, such that $d_1+\ldots +d_p=d$, if the following conditions hold true: \begin{itemize}
        \item for some $i=1,\ldots,p$, $f:\R^d \to \R$ is asymptotically separable onto $\R^{d_i} \times ( \bigtimes_{\substack{j=1,\ldots,p,\\j\neq i}} \R^{d_j})$, with non-negative definite factors $f_i:\R^{d_i}\to \R$ and $\hat f_i:\bigtimes_{\substack{j=1,\ldots,p,\\j\neq i}} \R^{d_j}\to\R$;
        \item $\hat f_i:\bigtimes_{j=1,\ldots,p,\\j\neq i} \R^{d_j}\to\R$ is asymptotically separable onto $(p-1)$ blocks $\bigtimes_{\substack{j=1,\ldots,p,\\j\neq i}} \R^{d_j}$.
    \end{itemize}  
    By iteration, \cref{thm:equivalence} extends to $p$-domain functionals, as in \cite{pdom24}, \begin{equation*} 
    Y_{t_1,\ldots,t_p} = \int_{t_1D_1 \times \ldots \times t_pD_p} H_2(B_x) dx, \qquad t_1,\ldots,t_p>0,
    \end{equation*}
    where $B=(B_x)_{x\in\R^d}$ is a real-valued, continuous, centered and unit-variance, stationary Gaussian field on $\R^d$, and its covariance function $C:\R^d\to\R$ is asymptotically separable onto $p$ blocks with non-negative definite factors $C_i:\R^{d_i}\to\R$, $i=1,\ldots,p$. In this case, $C^\textrm{Sep} = C_1\otimes\ldots\otimes C_p$.
\end{remark}

\begin{remark} The cumulants of any multiple integral $F=I_q(f)$, for some symmetric kernel $f$, $q\in \N$, satisfy a formula analogous to \cref{eq:kappaktilde}: $\kappa_1(F)=0$, $\kappa_2(F)=q!\|f\|^2$, and for $n\ge 3$ we have \begin{equation*}
    \kappa_n (f) = q! (n-1)! \sum c_q(r_2,\ldots,r_{n-1}) \scp{(\ldots ((f\widetilde\otimes_{r_2} f)\widetilde\otimes_{r_3} f) \ldots \widetilde\otimes f_{r_{m-2}})\widetilde\otimes_{r_{m-1}} f , f}_{\mathfrak H^{\otimes q}},
\end{equation*}
where the sum ranges over collections of integers $r_2,\ldots,r_{n-1}$, and $c_q(r_2,\ldots,r_{n-1})$ are explicit combinatorial constants; see \cite[Theorem 8.4.4]{bluebook} for details. Then, one would be tempted to introduce a higher-order asymptotic separability definition, and conclude that under the latter condition, \cref{thm:equivalence} continues to hold for functionals with Hermite rank $q\ge 3$. However, random variables in higher chaoses are not determined by their moments, and, equivalently, by their cumulants. See e.g. \cite{slud1993}. Then, a complete characterization of the limiting distribution for functionals with Hermite rank $R\ge 3$ is out of the scope of the present method. 
\end{remark}

\begin{remark}\label{rem:asympsep}
 \cref{thm:ltBREUERMAJOR,thm:ltBALTO,thm:ltROSENBLATT} imply also that Gneiting covariance functions are asymptotically separable in the sense of \cref{def:asympsep}. More precisely, let $C$ be a Gneiting covariance function with factors $C_1$ and $C_2$, as in \cref{eq:Cgneiting}. Then,
    \begin{itemize}
        \item if $C_1\in L^2(\R^{d_1})$, we have asymptotic separability with $C_1\otimes C_2^*$, where $C_2^*$ is a regularly varying covariance function with parameter $-\rho_2/2$;
        \item if $C_1\notin L^2(\mathbb R^{d_1})$ and, in addition, it is regularly varying with parameter $-\rho_1<0$ such that $2\rho_1<d_1$,  
        we have a.s. with $C_1\otimes C_2^*$, where $C_2^*$ is a regularly varying covariance function with parameter $-\rho_2({1-\rho_1/d_1})$.
    \end{itemize}
\end{remark}
    
\subsection{Plan of the paper} 

The structure of the paper is as follows. \cref{sec:prelim} is devoted to preliminary results: in \cref{sec:malliavinstein}, we introduce the framework of chaotic decompositions and the Nualart-Peccati Fourth Moment Theorem; in \cref{sec:rosenblatt}, we give explicit expressions for the characteristic functions of of Rosenblatt-type and $p$-domain Rosenblatt distributions; in \cref{sec:prelim misc}, we recall some standard results on regularly varying functions and covariograms, that we will extensively use in the following. \cref{sec:proof} is devoted to the proofs. In \cref{sec:range dep}, we investigate the property of short and long range dependance of the field, and we state and prove the variance growth rate in \cref{sec:variance}. Then, in \cref{subsec:proofLT} we prove our main results. Finally, Appendix \ref{sec:auxiliaryRV} collect a series of results on regularly varying functions integrated on growing domains.

\section{Preliminaries}\label{sec:prelim}

\subsection{Chaotic decompositions and fourth moment theorems}

\label{sec:malliavinstein}

In this subsection, we briefly recall the framework of isonormal Gaussian processes and Wiener chaoses, and explain how the Nualart--Peccati Fourth Moment Theorem can be used to establish
convergence in distribution for random variables living in a fixed Wiener chaos. We refer e.g. to
\cite{bluebook} for a comprehensive treatment.

\subsubsection{A Hilbert space associated with the covariance kernel.}
Let $B=(B_x)_{x\in\mathbb R^d}$ be a real-valued, continuous, centered, stationary Gaussian field
with unit variance and covariance function \begin{equation*}
    C(x-y)=\mathrm{Cov}(B_x,B_y), \qquad x,y\in\mathbb R^d.
\end{equation*}
We associate with $C$ a reproducing kernel Hilbert space (RKHS) $\mathcal H$, defined as follows. Let \begin{equation*}
    \mathcal H_0 := \mathrm{span}\{\,e_x(\cdot):=C(\cdot-x)\ :\ x\in\mathbb R^d\,\},
\end{equation*}
and endow $\mathcal H_0$ with the inner product uniquely determined by \begin{equation*}
    \langle e_x, e_y\rangle_{\mathcal H} := C(x-y), \qquad x,y\in\mathbb R^d,
\end{equation*}
extended bilinearly. Since $C$ is positive definite, this inner product is well defined. We denote
by $\mathcal H$ the completion of $\mathcal H_0$ with respect to the induced norm. By construction,
$\mathcal H$ is a real separable Hilbert space with reproducing kernel $C$. Now, consider an isonormal Gaussian process over $\mathcal H$, $X=\{X(h):h\in\mathcal H\}$,
that is, a centered Gaussian family satisfying \begin{equation*}
    \mathbb E[X(h)X(g)] = \langle h,g\rangle_{\mathcal H}, \qquad h,g\in\mathcal H.
\end{equation*}
With this construction, $(X(e_x))_{x\in\R^d}$ defines a Gaussian fields with the same distribution of $(B_x)_{x\in\R^d}$.
In what follows, we work within this representation.

\subsubsection{Wiener chaoses and multiple integrals.}
For $q\in\N$, let $H_q$ denote the $q$-th Hermite polynomial. For $q\in\N$, the $q$-th Wiener chaos $\mathcal H_q$
associated with $X$ is defined as the closed linear subspace of $L^2(\Omega)$ generated by random
variables of the form $H_q(X(h))$, with $h\in\mathcal H$ satisfying $\|h\|_{\mathcal H}=1$. One has
the orthogonal decomposition \begin{equation*}
    L^2(\Omega)=\bigoplus_{q=0}^{\infty} \mathcal H_q,
\end{equation*}
known as the Wiener chaos decomposition. For $q\in\N$, the multiple Wiener--It\^o integral of order $q$, denoted by $I_q$, can be first defined on $h^{\otimes q}$, $\|h\|_{\mathcal H}=1$, by \begin{equation*}
    I_q(h^{\otimes q}) := H_q(X(h)), \qquad \|h\|_{\mathcal H}=1,
\end{equation*}
and then extended by linearity and density (of ${\rm span}\{h^{\otimes q}\}$ in $\mathcal H^{\odot q}$)  to a linear isomorphism (up to factor $\sqrt{q!}$) from the symmetric tensor product
$\mathcal H^{\odot q}$ onto $\mathcal H_q$,  satisfying \begin{equation*}
    \mathbb E[I_q(h^{\otimes q})I_q(g^{\otimes q})]=q! \, \scp{h,g}^q_{\mathcal H}, \qquad \forall h,g \in\mathcal{H^{\odot  q}}\,.
\end{equation*}
In particular, for any measurable set $D\subset\mathbb R^d$ with finite Lebesgue measure, the random
variable \begin{equation*}
    \int_D H_q(B_x)\,dx
\end{equation*}
belongs to the $q$-th Wiener chaos. Indeed, since $B_x \eqlaw X(e_x)$, defining
\begin{equation}
    \label{eq: fD}
f_{D,q} := \int_D e_x^{\otimes q}\,dx  \in\ \mathcal H^{\odot q},
\end{equation}
we have \begin{equation*}
    \int_D H_q(B_x)\,dx \eqlaw I_q(f_{D,q}).
\end{equation*}
Moreover, the orthogonal decomposition of $\varphi\in L^2(\gamma)$, where $\gamma$ denotes the standard Gaussian measure in $\R$, with respect to Hermite polynomials, written $\{H_q:q\ge 1\}$, induces a Wiener chaos decomposition for functionals of the form $Y=\int_D\varphi(B_x)\,dx \in L^2(\Omega)$, of the form
\begin{equation}
\label{eq: Y deco}
    Y=\int_D\varphi(B_x)\,dx \overset{L^2(\Omega)}{=}\sum_{q\in \N}Y[q]\,\,,\qquad Y[q]:=a_q\int_D H_q(B_x)\,dx \overset{d}{=} I_q(a_q f_{D,q})\,.
\end{equation}
We will refer to $Y[q]$ as the $q$th chaotic component of $Y$. For the variance, we have \begin{equation*}
    \Var(Y) = \sum_{q\in\N}\Var(Y[q]) = \sum_{q\in\N}q!a_q^2\int_D\int_DC(x-y)^q\,dx\,dy\,.
\end{equation*}
\subsubsection{Contractions and the Fourth Moment Theorem.}
For $q\in\mathbb N$, $r=1,\ldots,q-1$ and $h,g$ symmetric functions with unit norm in $\mathcal H$, we can define the $r$th contraction of $h^{\otimes q}$ and $g^{\otimes q}$ as the (generally non-symmetric) element of $\mathcal H^{\otimes 2q-2r}$ given by \begin{equation*}
    h^{\otimes q} \otimes_r g^{\otimes q} = \scp{h,g}_{\mathcal H}^{r} \,\,h^{\otimes q-r} \otimes g^{\otimes q-r}.
\end{equation*}
We then extend the definition of contraction to every pair of elements in $\mathcal H^{\odot q}$. We will denote the norm in this space by $\|\cdot\|_q$. A sharp criterion for asymptotic normality inside a fixed Wiener chaos is provided by the following result, the celebrated Nualart--Peccati Fourth Moment Theorem, see
\cite{cltNuaPec}.

\begin{theorem}[Fourth Moment Theorem]
\label{thm:4th}
Let $q\ge2$ and let $(h_t)_{t>0}\subset \mathcal H^{\odot q}$ be such that \begin{equation*}
    \mathbb E[I_q(h_t)^2]\longrightarrow 1 \qquad \text{as } t\to\infty.
\end{equation*}
Then the following statements are equivalent:
\begin{enumerate}
\item $I_q(h_t)$ converges in distribution to a standard Gaussian random variable;
\item $\mathbb E[I_q(h_t)^4]\longrightarrow 3$;
\item for every $r=1,\dots,q-1$, $\|h_t\otimes_r h_t\|_{2q-2r} \longrightarrow 0$ as $t\to\infty$.
\end{enumerate}
\end{theorem}
\noindent
Remark that when $h = f_{D,q}$, $q\ge1$, for some domain $D\subset \R^d$, as in \eqref{eq: fD}, we may write, for any $r=1,\ldots,q-1$, \begin{equation*}
    \|h\otimes_r h\|^2_{2q-2r} = a_q^4 \, \int_{D^4} C(x-y)^{q-r} C(y-z)^{r} C(z-v)^{q-r} C(v-x)^{r} \, dx\, dy\, dz \, dv.
\end{equation*} 
When a finite-chaos domination does not occur, 
one may rely on the following result,
which provides sufficient conditions for asymptotic normality starting from a possibly
infinite chaos decomposition; see, for instance, \cite[Theorem~6.3.1]{bluebook}.

\begin{theorem}\label{thm 631}
Let $(F_t)_{t\ge1}$ be a sequence in $L^2(\Omega)$ such that $\mathbb{E}[F_t]=0$
for all $t$. Consider the chaos expansion \begin{equation*}
    F_t = \sum_{q=1}^{\infty} I_q(f_{t,q}),
\qquad f_{t,q} \in \mathcal{H}^{\odot q}, \; q \ge 1,
\end{equation*}
and suppose in addition that:
\begin{enumerate}
\item for every fixed $q \ge 1$, there exists $\lim_{t\rightarrow\infty} q!\,\|f_{t,q}\|^2_{\mathcal{H}^{\otimes q}} \longrightarrow \sigma_q^2\ge0$ as $t\to \infty;$
\item $\sigma^2 := \sum_{q=1}^{\infty} \sigma_q^2 < \infty$;
\item for all $q \ge 2$ and $r = 1,\dots,q-1$, $\|f_{t,q} \otimes_r f_{t,q}\|_{\mathcal{H}^{\otimes (2q-2r)}}
\longrightarrow 0$ as $t \to \infty$;
\item $\lim_{N \to \infty}
\sup_{t \ge 1}
\sum_{q=N+1}^{\infty}
q!\,\|f_{t,q}\|^2_{\mathcal{H}^{\otimes q}}
= 0$.
\end{enumerate}
Then, $F_t \convlaw \mathcal{N}(0,\sigma^2)$ as $t\to\infty$.
\end{theorem}

\subsubsection{Cyclic products, and characteristic functions in the second chaos.}
\label{subsec:fredholmappears}

\begin{definition}[Cyclic product]
For $f:\R^d\to\R$ and $k\ge2$ we define the \emph{cyclic product}
\begin{equation}
\label{def:cyclic}
f^{\circ k}(x^{(1)},\dots,x^{(k)})
:=
\prod_{i=1}^k f\left(x^{(i)}-x^{(i+1)}\right),
\qquad x^{(k+1)}:=x^{(1)}.
\end{equation}
\end{definition}
\noindent
We refer to \Cref{sec:auxiliaryRV} for further extensions and properties. As a consequence of \cite[Proposition 2.7.13]{bluebook}, we have the following.

\begin{proposition}
\label{fredholm does the job}
    Recall \eqref{eq: fD}. The characteristic function of $I_2(f_{D,2})$ admits the following representation \begin{equation*}
        \phi(\xi) = \exp\left\{ \frac12 \sum_{m=2}^{\infty} \frac{(2i\xi)^m}{m} c_m \right\},
    \end{equation*}
    where  \begin{equation*}
        c_m = \int_{D^m} C^{\circ m} (x_1,\ldots, x_m) \, dx_1 \ldots dx_k.
    \end{equation*}
    Moreover, its distribution is moment-determined, and the cumulants of $I_2(f_{D,2})$, written $\{\kappa_k: k\ge 1\}$ are $\kappa_1 = 0$, and, for $k\ge 2$,
    \begin{equation*}
        \kappa_k = 2^{k-1}(k-1)! \, c_k.
    \end{equation*}
\end{proposition}

\subsection{Rosenblatt distributions}
\label{sec:rosenblatt}

In this section, we recall the definition of $p$-domain Rosenblatt random variables, as in \cite{pdom24}. We also provide explicit expressions for their characteristic function.

\subsubsection{$2$-domain Rosenblatt distribution}
\label{def:2domainRosenblatt}

For $i=1,2$, consider a convex body $D_i\in \R^{d_i}$, and $d:=d_1+d_2$. Consider positive parameters $\alpha,\beta \in \R$ such that $\alpha\in (0,d_1/2)$ and $\beta\in(0,d_2/2)$. Let $H=H_{2,D_1\times D_2}^{\alpha,\beta}$ be a $2$-domain Rosenblatt random variables, see \cite[Remark 1.5, Theorem 1.6]{pdom24}\footnote{When $R=2$, in \cite[Theorem 1.6]{pdom24}, we may relax the hypothesis on the existence of an absolutely continuous spectral measure. See \cref{rem: no spectral reg var}.}.

\begin{proposition}
\label{prop:rosenblatt characteristic}
The characteristic function of a $2$-domain Rosenblatt random variable $H$ with positive parameters $\alpha,\beta\in\R$ such that $\alpha\in (0,d_1/2)$ and $\beta\in(0,d_2/2)$, and convex bodies $D_i\subset \R^{d_i}$, $i=1,2$, is \begin{equation*}
    \varphi(\xi) = \E[\exp(i\xi H)] = \exp\left(\frac12\sum_{k=2}^\infty \frac{(2i\xi)^k}{k} c_k^{\alpha,\beta} \right),
\end{equation*}
where \begin{multline}\label{eq:ckRosen}
    c_k^{\alpha,\beta} = (2\sigma^2)^{-k/2} \, \int_{(D_1)^k} \| x_1-x_2 \|^{-\alpha} \ldots \|x_k-x_1\|^{-\alpha} \, dx_1 \ldots dx_k \\
    \times \int_{(D_2)^k} \| y_1-y_2 \|^{-\beta} \ldots \|y_k-y_1\|^{-\beta} \, dy_1 \ldots dy_k,
\end{multline}   
and 
\begin{equation*}
    \sigma^2 := \int_{(D_1)^2} \| x_1-y_1 \|^{-2\alpha} dx_1dy_1 \times \int_{(D_2)^2} \| x_2-y_2 \|^{-2\beta} dx_2dy_2.
\end{equation*}
Moreover, its distribution is moment-determined, and the cumulants of $X$, written $\{\kappa_k: k\ge 1\}$ are $\kappa_1 = 0$, $\kappa_2(H)=1$, and for $k\ge 3$, $\kappa_k(H) = 2^{k-1}(k-1)! c_k^{\alpha,\beta}$. 
\end{proposition}
\begin{remark}
    The factor $(2\sigma)^{-k/2}$ appearing in \cref{eq:ckRosen} is due to the fact that a $p$-domain Rosenblatt random variable, as defined in \cite{pdom24}, is normalised to have unit-variance.
\end{remark}
\begin{proof}[Proof of \cref{prop:rosenblatt characteristic}]
Let $B=\{B(x_1,x_2):(x_1,x_2)\in\R^{d_1}\times\R^{d_2}\}$ be a centered, unit-variance, stationary
Gaussian field with covariance functions $C=C_1\otimes C_2:\R^{d_1}\times \R^{d_2}\to\R$, where $C_1$ and $C_2$ are regularly varying at $\infty$ covariance functions with parameter $-\alpha\in(-d_1/2,0)$ and $-\beta\in(-d_2/2,0)$, respectively. Thanks to \cite[Theorem 1.6]{pdom24}\footnote{When $R=2$, we may relax the hypothesis on the existence of an absolutely continuous spectral measure. See \cref{rem: no spectral reg var}.}, as $t_1,t_2\to\infty$, 
\begin{equation*}
    Y(t_1,t_2) := \frac{\int_{t_1D_1\times t_2D_2} H_2(B_{x_1,x_2}) \, dx_1\, dx_2}{ \sqrt{\Var\left(\int_{t_1D_1\times t_2D_2} H_2(B_{x_1,x_2}) \, dx_1\, dx_2\right)}} \convlaw H,
\end{equation*}
where $H$ is a $2$-domain Rosenblatt random variable with indices $\alpha, \beta$ and domains $D_1,D_2$. Therefore, the cumulants of $H$, defined by $\{\kappa_k:k\ge 1\}$, are $\kappa_1=0$, $\kappa_2=1$  ($H$ has zero mean and unit variance), and, for $k\ge3$, we need to compute the limit 
\begin{equation*}
\kappa_k = 2^{k-1} (k-1)!  \lim_{t_1,t_2\to\infty} \kappa_k(t_1,t_2)  ,
\end{equation*}
where $\kappa_k(t_1,t_2)$ denotes the $k$th cumulant of $Y(t_1,t_2)$. We may identify $Y(t_1,t_2)$ with a suitable double Wiener-Ito integral as follows: \begin{equation*}
    Y(t_1,t_2) \eqlaw \frac{I_2\left( {f_{t_1D_1\times t_2D_2}}\right)}{ \sqrt 2 \|f_{t_1D_1\times t_2D_2,2}\|_{\mathcal H^{\otimes 2}} } ,
\end{equation*} 
where $f_{t_1D_1\times t_2D_2,2}$ is constructed as in \cref{eq: fD}, on a Hilbert space $\mathcal H$ with inner product induced by the covariance $C$. Applying \cref{fredholm does the job} and since $C$ is in tensor form, we have that \begin{equation*}
    \kappa_k(t_1,t_2) = 2^{k-1} (k-1)! \, 2^{-k/2} \prod_{i=1,2} \frac{\int_{(t_iD_i)^k} C_i^{\circ k}(x_i^{(1)},\ldots , x_i^{(k)}) \, dx_i^{(1)}\ldots dx_i^{(k)}}{ \left(\int_{(t_iD_i)^2} C_i(x_i-y_i)^2 \, dx_i\,dy_i\right)^{k/2}}.
\end{equation*}
By \cite[Equation (9)]{Maini2024}, we have that $\int_{(t_iD_i)^2} C(x_i-y_i)^2 \, dx_i\,dy_i \sim \sigma_i^2 t_i^{2d_i}c_i(t_i)^2$, where \begin{equation*}
     \sigma_1^2 := \int_{(D_1)^2} \| x_1-y_1 \|^{-2\alpha} dx_1dy_1, \qquad \sigma_2^2 := \int_{(D_2)^2} \| x_2-y_2 \|^{-2\beta} dx_2dy_2.
\end{equation*} 
Applying \cref{lem:cactus}, we may conclude.
\end{proof}

\begin{remark}
    The previous result can be easily generalised to $p$-domain Rosenblatt random variables $H_{2,D_1\times \ldots\times D_p}^{\alpha_1,\ldots,\alpha_p}$, indexed over domains $D_i\subset \R^{d_i}$, and positive parameters $\alpha_i<d_i/2$, simply taking \begin{equation*}
        c_k^{\alpha_1,\ldots,\alpha_p} = 2^{-k/2} \prod_{i=1}^p \int_{(D_i)^k} \| x_1-x_2 \|^{-\alpha_i} \ldots \|x_k-x_1\|^{-\alpha_i} \, dx_1 \ldots dx_k .
    \end{equation*}
\end{remark}

As a corollary, we may give an explicit expression for the characteristic function of a Rosenblatt-type random variable, as introduced in \cite{LO14}.

\begin{proposition}
\label{prop:rosencumulants}
Consider $\Delta\subset \R^d$, and $\alpha\in (0,d/2)$. Let $X:=X_2^\alpha(\Delta)$ be a Rosenblatt-type distribution of parameter $\alpha$. The characteristic function of $X$ admits the following expression: \begin{equation*}
    \varphi(\xi) = \E[\exp(i\xi X)] = \exp\left(\frac12\sum_{k=2}^\infty \frac{(2i\xi)^k}{k} c_k^\alpha\right),
\end{equation*} 
where the coefficients $c_k^\alpha$, $k\ge 2$, are as follows:
\begin{equation*}
    c_k^\alpha 
    = \int_{\Delta^k} \frac1{\|x^{(1)}-x^{(2)}\|^{\alpha}}\frac1{\|x^{(2)}-x^{(3)}\|^{\alpha}}\ldots \frac1{\|x^{(k)}-x^{(1)}\|^{\alpha}} \, dx^{(1)} \ldots dx^{(k)}.
\end{equation*}
Moreover, its distribution is moment-determined, and the cumulants of $X$, written $\{\kappa_k: k\ge 1\}$ are $\kappa_1 = 0$, and for $k\ge 2$
    \begin{equation*}
        \kappa_k = 2^{k-1}(k-1)! \, c_k^\alpha.
    \end{equation*}
\end{proposition}

\begin{remark}
    If $d=1$, a different proof can be found in \cite{VT13cumuRosenblatt}. We anticipate that an analogous argument can be extended to $d\ge2$.
\end{remark}

\subsection{Miscellanea}
\label{sec:prelim misc}
\subsubsection{Covariograms}
\begin{definition}[Covariogram]\label{def:covariogram}
We define the covariogram $g_d : \R^d \to \R_{\ge 0}$ of a subset $D\subset \R^d$ as 
    \begin{equation*}
        g_d (D, z) := \vol( D \cap (D+z)), \qquad \forall z\in\mathbb R^{d_i}.
    \end{equation*}
\end{definition}

\begin{proposition}\label{prop:covariogram} \begin{itemize}
\item[(i)] For all $z\in\R^d$, we have $g_d(D,z)\le \vol(D)$.
\item[(ii)] For all $z\in\R^d$ and $t>0$, we have $g_d(tD,z) = t^d g_d(D,z/t)$.
    \item[(iii)] If $D$ is convex, for all $s,t\in\R$ such that $s<t$ we have \begin{equation*}
        g_d(D,s z) \ge  g_d(D,t z) \qquad \forall z\in\R^d.
    \end{equation*}
    \item[(iv)] Let $D\subset \R^d$ have finite volume. The sequence $\{g_n:\R^d \to \R_{\ge}: n\in\N\}$, defined by \begin{equation*}
        g_n: z\mapsto g_d(D,z/n), 
    \end{equation*}
    converges uniformly on any compact set to $\vol(D)$.
\end{itemize}
\end{proposition}
\begin{proof}
\emph{(i)-(ii)} Evident. \emph{(iii)} If $x\in D\cap D+tz$ and $D$ is convex, we have that $\delta (x-tz) + (1-\delta) x = x -\delta tz \in D$, for all $\delta \in [0,1]$. Taking $\delta=s/t$, we conclude. \emph{(iv)} Since $D$ has finite volume, we have that $g_n$ converges pointwise to $\vol(D)$ on $\R^d$. Moreover, for any fixed $z\in\R^d$ the convergence is monotonic increasing by \emph{(iii)}. Uniform convergence over compact sets then follows by Dini's theorem.
\end{proof}

\subsubsection{Potter bounds}

A classical reference for the next result is the monograph \cite{bingham}. See also \cite[Lemma 2.1 (iv)]{MainiRossiZheng2025}.

\begin{proposition}[Potter bounds for regularly varying functions]\label{prop:potter} 
Let $f:\R\to\R^+$ be regularly varying at $+\infty$ with index $\rho$. Then, for any $\delta > 0$ and for any $A > 1$, there exists some constant  
$X = X(A, \delta)$ such that, for any $x, y \in [X, \infty)$, \begin{equation*}
A^{-1} \min \left\{ (y/x)^{\rho + \delta}, (y/x)^{\rho - \delta} \right\} 
\le \frac{f(y)}{f(x)} 
\le A \max \left\{ (y/x)^{\rho + \delta}, (y/x)^{\rho - \delta} \right\}.
\end{equation*}
In particular, this implies that as $x\rightarrow+\infty$
\begin{equation}
    \label{eq:regvar 0 e infty}
    f(x)x^{-\rho+\delta}\rightarrow +\infty\,\,,\quad\quad\quad f(x)x^{-\rho-\delta}\rightarrow 0\,\,.
\end{equation}
\end{proposition}

An important property of regularly varying functions that will be used in the sequel is uniform convergence on compact sets in the following sense. For a proof, see e.g. \cite[Theorem 1.5.2] {bingham}.

\begin{proposition}[Uniform convergence theorem for regularly varying functions] 
\label{prop:UCTregular}
Let $f:\R\to\R^+$ be regularly varying at $+\infty$ with index $\rho\le0$. Then, the following convergence holds uniformly on compact sets for $x\in (0,+\infty)$: \begin{equation*}
    \lim_{\lambda\to\infty} \frac{f(\lambda x)}{f(\lambda)}\rightarrow x^{\rho}\,. 
\end{equation*}
In particular, if $\lambda_t\rightarrow\infty$ and $u_t\rightarrow u\in(0,\infty)$ as $t\to\infty$, then we have \begin{equation}\label{eq: uct equation}
    \lim_{t\to\infty} \frac{f(\lambda_t\,u_t)}{f(\lambda_t)}\rightarrow u^{\rho}\,.
\end{equation}
\end{proposition}

\section{Proofs}\label{sec:proof}

In this section, preliminary to the proofs of our main results, we investigate the property of short and long range dependance of the field (\cref{sec:range dep}), and we state and prove the variance growth rate (\cref{sec:variance}). Then, in \cref{subsec:proofLT} we prove our main results.

\subsection{Range dependence}

\label{sec:range dep}

We investigate the properties of range dependence of $\f(B_x)$, and describe the variance asymptotics of \eqref{eq:Yt}.

\begin{definition}
    Fix $d\ge 1$. Let $Z=(Z_x)_{x\in\R^d}$ be a continuous and stationary random field on $\R^d$. We say that it exhibits short-range dependence if \begin{equation*}
        \int_{\R^d} |\Cov(Z_x,Z_0)| \, dx < \infty,
    \end{equation*}
    and long-range dependence otherwise. If $d=1$, we say that it exhibits short (or long) memory.
\end{definition}

\begin{proposition} \label{prop:GneitingRD} Consider the field $\f(B)=(\f(B_x))_{x\in\R^d}$, where $B=(B_x)_{x\in\R^d}$ is a centered, unit-variance, stationary and continuous Gaussian field; $\varphi:\R\to\R$ is such that $\E[\varphi(B_0)^2]<\infty$ and has Hermite rank $R\ge 1$; and the covariance function $C$ of $B$ satisfies \Cref{ass:gneiting}. Then, $\f(B)$ exhibits long-range dependence if $R=1$, or $R\ge 2$ and at least one of the following properties hold: 
    \begin{equation} \label{eq:notbm}
         C_1 \notin L^R(\mathbb R^{d_1}) \quad \text{or} \quad C_2 \notin L^{R-1}(\mathbb R^{d_2}).
    \end{equation}
\end{proposition}

\begin{corollary} If $C_1$ and $C_2$ are regularly varying with parameters $-\rho_1<0$ and $-\rho_2<0$ (resp.), then the field $\f(B)=(\f(B_x))_{x\in\R^d}$ exhibits long-range dependence if and only if $R=1$, or $R\ge 2$ and \begin{equation*}
        \rho_1 \le \frac{d_1}{R} \quad \text{or} \quad \rho_2 \le \frac{d_2}{R-1}.
    \end{equation*}
\end{corollary}
\begin{remark}
   If the Hermite rank of $\varphi$ is $R=1$, then we always have long-range dependence. Moreover, we can see that the roles of $C_1$ and $C_2$ are not symmetric, unlike in the case of separable covariance functions. See for comparison \cite[Theorem 3.6]{pdom24}.
\end{remark}

\begin{proof}[Proof of \cref{prop:GneitingRD}] As the fields $H_p(B)$ and $H_q(B)$ are uncorrelated if $p\neq q$, and $|C|^q\le |C|^p$ for any $p\le q$, it's enough to check if $C\in L^R(\R^d)$, where $R$ is the Hermite rank of $\varphi$. If $C$ belongs to the Gneiting class with factors $C_1$ and $C_2$, we have short-range dependence if, up to the change of variables $x_1\mapsto x_1C_2(x_2)^{1/d_1}$, we have \begin{multline*}
    \int_{\mathbb R^{d_2}} |C_2(x_2)|^R \int_{\mathbb R^{d_1}} |C_1((x_1)C_2(x_2)^{1/d_1})|^R dx_1 dx_2 \\
    = \int_{\mathbb R^{d_2}} |C_2(x_2)|^{R-1} \int_{\mathbb R^{d_1}} |C_1(x_1)|^R dx_1dx_2 < \infty.
\end{multline*}
The previous holds if $R>1$, $C_2\in L^{R-1}(\mathbb R^{d_2})$ and $C_1\in L^{R}(\mathbb R^{d_1})$. If \Cref{ass:RV} holds, namely, $C_1$ and $C_2$ are regularly varying with parameters $-\rho_1<0$ and $-\rho_2<0$, resp., we have short-range dependence when $R>1$ and $\rho_1 > \frac{d_1}{R}$ and $\rho_2 > \frac{d_2}{R-1}$.
\end{proof}

\subsection{Variance asymptotics}\label{sec:variance}

\begin{proposition}\label{prop:variance} Consider $Y(t)$ as in \eqref{eq:Yt}. Suppose Assumptions \ref{ass:gneiting} and \ref{ass:KEY} hold. Suppose that $\f=H_R$, $R\ge 2$. Then, as $t \to\infty$: \begin{enumerate}
    \item[(1)] if $C\in L^{{R}}(\R^d)$, then 
    \begin{equation}\label{eq:varianceBM}
        \Var(Y(t)) \sim \ell \, t_1(t)^{d_1}t_2(t)^{d_2},
    \end{equation}
    where $\ell = R! \, \vol(D_1) \vol(D_2) \, \int_{\R^{d_1}} C_1(x_1)^R dx_1 \, \int_{\R^{d_2}} C_2(x_2)^{R-1} dx_2$;
    \item[(2)] if $C_1\in L^{{R}}(\R^{d_1})$ and $C_2\notin L^{{R-1}}(\mathbb R^{d_2})$, but is regularly varying with index $-\rho_2<0$ such that $\rho_2<d_2/(R-1)$, then\footnote{Remark that the contribution of $t_2$ to the variance is larger than the volume, that is, we have $2d_2-(R-1)\rho_2>d_2$ by assumption on $\rho_2$. Hence, we are in a long-range dependent setting. Same holds for case \emph{(3)} and \emph{(4)}.}   \begin{equation}\label{eq:varianceNOTBM2}
        \Var(Y(t)) \sim t_1(t)^{d_1} \, t_2(t)^{2d_2-({{R-1}})\rho_2} \, L_2^\prime(t_2(t)),
    \end{equation}
    where $L_2^\prime$ is a slowly varying function.
    \end{enumerate}
Suppose now that $C_1\notin L^{{R}}(\mathbb R^{d_1})$, but is regularly varying with index $-\rho_1<0$ such that $\rho_1<d_1/R$. Then, \begin{enumerate}[resume]
    \item[(3)] if $C_2\in L^{R(1-\rho_1/d_1)}(\R^{d_2})$\footnote{In particular, this happens when $C_2\in L^{{R-1}}(\R^{d_2})$, since $d_2>\frac{d_1d_2}{{{R}}(d_1-\rho_1)}$ when ${{R}}\rho_1<d_1$. See also \cref{rem:newrange}.}, then \begin{equation}\label{eq:varianceNOTBM3}
        \Var(Y(t)) \sim t_1(t)^{2d_1-{{R}}\rho_1}t_2(t)^{d_2} L_1^\prime(t_1(t)),
    \end{equation}
    where $L_1^\prime$ is a slowly varying function;
    \item[(4)] if $C_2\notin L^{R(1-\rho_1/d_1)}(\R^{d_2})$, but in addition satisfies \Cref{ass:RV} with parameter $-\rho_2<0$ such that $\rho_2 < \frac{d_1d_2}{{{R}}(d_1-\rho_1)}$ (in particular, if $C_2\notin L^{{R-1}}(\mathbb R^{d_2})$), then, \begin{equation}\label{eq:varianceNOTBM4}
        \Var(Y(t)) \sim t_1(t)^{2d_1-{{R}}\rho_1}t_2(t)^{2d_2 - {{R}}\rho_2(1-\frac{\rho_1}{d_1})} L_1^{\prime\prime}(t_1(t))L_2^{\prime\prime}(t_2(t)),
    \end{equation}
    where $L_1^{\prime\prime}$ and $L_2^{\prime\prime}$ are two slowly varying functions.
    \end{enumerate}
\end{proposition}

\begin{remark}
\label{rem:R1variance}
    When $R=1$, the contribution of the second variable $t_2$ is always of order $\vol(t_2D_2)^2$. Then, \begin{equation*}
        \Var(Y(t)) \sim t_2^{2d_2} \vol(D_2)^2 \, \Var\left( \int_{t_1D_1} B_{x_1}^\prime \,  dx_1 \right), \qquad B_{x_1}^\prime:=B_{x_1,x_2}, 
    \end{equation*}
    for a fixed $x_2\in\R^{d_2}$. Then it may be understood using classical results for additive functionals of Gaussian fields: see, e.g., \cite{BM83,T79,DM79,MN22}.
\end{remark}

\begin{remark}
    When $C_1$ and $C_2$ are asymptotically power laws, that is, their slowly varying component is asymptotically constant, the slowly varying functions appearing in the previous equations become constants that can be explicitly computed.
\end{remark}

\begin{remark}[New intermediate range] \label{rem:newrange}
    The critical exponent arising in the third and fourth cases (corresponding to the lower left and lower right regions in Figure~\ref{fig:lt} when $R=2$) can be read as follows. When $C_1\notin L^{R}(\mathbb R^{d_1})$ (left region), the parameter of regular variation $-\rho_1$ is such that $d_1-{R}\rho_1 \ge 0$. This implies that \begin{equation*}
         (R-1)\rho_2 \le{(R-1)}\rho_2 + \rho_2\left(\frac{d_1 - R \rho_1}{d_1}\right) =  {R} \underbrace{\rho_2\left(1-\frac{\rho_1}{d_1}\right)}_{=:\rho_2^*} \le R\rho_2.
    \end{equation*}
    We can interpret the exponent $\rho_2^*$ as the \emph{effective regular variation parameter} of $C_2$, meaning that the contribution to the variance given by the correlation along the second block is the same as in a separable model $C_1\otimes C_2^*$, where $C_2^*$ is a regularly varying covariance function with parameter $-\rho_2^*$. 
    Conversely, when $C_1\in L^{R}(\mathbb R^{d_1})$ (right region), the effective regular variation parameter is $-\rho_2(1-1/R)$. 
    In both cases, the model built from a Gneiting covariance function having regularly varying factors $C_1$ and $C_2$, with parameters $-\rho_1$ and $-\rho_2$, behaves as a separable model with regularly varying factors $C_1$ and $C_2^*$, with parameters $-\rho_1$ and $-\rho_2^*$, with a longer range dependence along the second block: $\rho_2^* < \rho_2$.
\end{remark}

We are now in a position to prove \cref{prop:variance}. The proof is split into three parts: the proof of the first two cases, when $C_1\in L^R(\R^{d_1})$; the proof of the fourth case; a sketch of the proof of the third one, since it's analogous to the one of the fourth case. 

\textbf{Notation.} In the following, we will write $t_1$ and $t_2$ omitting the dependence on $t$. Recall that $t_1,t_2\to \infty$ as $t\to\infty$, as in \Cref{ass:KEY}.

\begin{proof}[Proof of \cref{prop:variance} \emph{(1)} and \emph{(2)}]
Consider the covariogram $g_{d_i}(D_i,x_i)$, $x_i\in\R^{d_i}$ for $i=1,2$, as in \cref{def:covariogram}. By change of variables $x_i-y_i\mapsto z_i$, $i=1,2$, and \cref{prop:covariogram} (i), we can write 
\begin{multline}
\label{eq:var}
        \frac{ \Var(Y(t))}{t_1^{d_1}  t_2^{d_2} }   = R! \int_{\mathbb R^{d_2}} c_2(\|z_2\|)^{R-1} g_{d_2} \left(D_2 , \frac{z_2}{t_2} \right) \\ \times \int_{\mathbb R^{d_1}} c_1(\|z_1\|)^{R} g_{d_1}\left( D_1, \frac{z_1}{t_1c_2(\|z_2\|)^{1/d_1}} \right) dz_1 \, dz_2.
    \end{multline}
We want to show that $\eqref{eq:var}\to \ell \in (0,\infty)$ as $t\to\infty$. By \cref{prop:covariogram} (iv), we have the pointwise convergence \begin{multline*}
    c_2(\|z_2\|)^{R-1} g_{d_2} \left(D_2 , \frac{z_2}{t_2} \right) c_1(\|z_1\|)^{R} g_{d_1}\left( D_1, \frac{z_1}{t_1c_2(\|z_2\|)^{1/d_1}} \right) \\ \longrightarrow c_2(\|z_2\|)^{R-1} \vol(D_2) c_1(\|z_1\|)^{R} \vol(D_1),
\end{multline*}
for any $(z_1,z_2)\in\R^{d_1}\times \R^{d_2}$. Since the covariograms are uniformly bounded by $g_{d_i}(D_i,\cdot)\le \vol(D_i)$, $i=1,2$, and $C_1\in L^R(\R^{d_1})$ and $C_2\in L^{R-1}(\R^{d_2})$, we may conclude by dominated convergence that $\eqref{eq:var} \to R! \vol(D_1)\vol(D_2) \|C_1\|^R_{L^R(\R^{d_1})} \|C_2\|^{R-1}_{L^{R-1}(\R^{d_2})}$. 

\medskip
\noindent
Assume now that $C_2\notin L^{R-1}(\R^{d_2})$, but is regularly varying with index $-\rho_2<0$ such that $\rho_2<d_2/(R-1)$. By change of variables  $x_2,y_2\mapsto t_2x_2,t_2y_2$, and $x_1-y_1\mapsto z_1$, which introduces the covariogram $g_{d_1}(D_1,\cdot)$ as in \cref{def:covariogram}, we find \begin{multline*}
    \frac{ \Var(Y(t))}{t_1^{d_1}  t_2^{2d_2} c_2(t_2)^{R-1}} = R! \int_{(D_2)^2} \frac{c_2(t_2\|x_2-y_2\|)^{R-1}}{c_2(t_2)^{R-1}} \\ \times  \int_{\mathbb R^{d_1}} c_1(\|z_1\|)^{R} g_{d_1}\left( D_1, \frac{z_1}{t_1c_2(t_2\|x_2-y_2\|)^{1/d_1}} \right) dz_1 \, dx_2 \, dy_2 . 
\end{multline*} 
Since $D_1$ is convex, by \cref{prop:covariogram} (iii), we have \begin{equation*}
        g_{d_1}\left( D_1, \frac{z_1}{t_1c_2(t_2\|x_2-y_2\|)^{1/d_1}} \right) \ge g_{d_1}\left( D_1, \frac{z_1}{t_1c_2(t_2\diam(D_2))^{1/d_1}} \right), \quad \forall x_2,y_2\in D_2.
    \end{equation*} 
Moreover, for any $\epsilon>0$, we may find $t$ (hence, $t_1,t_2$) large enough such that\footnote{It is not immediate: by positivity, the left-hand side is larger than $\int_{K_\epsilon} c_1(\|z_1\|)^{R} g_{d_1}\left( D_1, \frac{z_1}{t_1c_2(t_2M)^{1/d_1}} \right) dz_1$,
where $K_\epsilon$ is a large enough compact set and $M:=\diam (D_2)$. By compactness, we have the upper bound $g_{d_1}\left( D_1, \frac{z_1}{t_1c_2(t_2\diam(D_2))^{1/d_1}} \right) \ge \vol(D_1) - \epsilon$, for any $z_1\in K_\epsilon$. Since $C_1\in L^R(\R^{d_1})$, we may chose $K_\epsilon$ so that $\|C_1\|_{L^R(\R^{d_1}\setminus K_\epsilon)}$ is small, which allow us to conclude.} \begin{equation*}
    \int_{\mathbb R^{d_1}} c_1(\|z_1\|)^{R} g_{d_1}\left( D_1, \frac{z_1}{t_1c_2(t_2\diam(D_2))^{1/d_1}} \right) dz_1 \ge \int_{\mathbb R^{d_1}} c_1(\|z_1\|)^{R} dz_1 \, (\vol(D_1) - \epsilon).
\end{equation*}
Then, applying \cref{lem:cactus} ($k=2$), we find that  \begin{equation*}
    \lim_{t\to\infty} \int_{(D_2)^2} \frac{c_2(t_2\|x_2-y_2\|)^{R-1}}{c_2(t_2)^{R-1}} dx_2 \, dy_2 = \int_{(D_2)^2} \|x_2-y_2\|^{-\rho_2(R-1)} dx_2 \, dy_2 =: \ell^{\prime \prime} < \infty,
\end{equation*} 
implying in turn that
\begin{align*}
    & \lim_{t\to\infty} \frac{ \Var(Y(t))}{t_1^{d_1}  t_2^{2d_2} c_2(t_2)^{R-1}}  \ge  R! (\vol(D_1)-\epsilon) \|C_1\|^R_{L^R(\R^{d_1})} \, \ell^{\prime \prime};\\
    & \lim_{t\to\infty} \frac{ \Var(Y(t))}{t_1^{d_1}  t_2^{2d_2} c_2(t_2)^{R-1}}  \le R! \vol(D_1) \|C_1\|^R_{L^R(\R^{d_1})} \, \ell^{\prime\prime}.
\end{align*} 
By arbitrariness of $\epsilon$, we conclude.
\end{proof}

\begin{remark}\label{rem:useful var C1 int}
As soon as $C_1\in L^R(\R^{d_1})$, with no further assumptions on $C_2$, when $t\to\infty$, \begin{multline*}
    \Var(Y(t)) \asymp t_1^{d_1} t_2^{d_2} \int_{\mathbb R^{d_2}} c_2(\|z_2\|)^{R-1} g_{d_2} \left(D_2 , \frac{z_2}{t_2} \right) dx_2 \, dy_2 
     = t_1^{d_1} \int_{(t_2D_2)^2} C_2(x_2-y_2)^{R-1} dx_2 \, dy_2.
\end{multline*}
\end{remark}

\begin{proof}[Proof of \cref{prop:variance}: case \emph{(4)}]

If $C_1$ and $C_2$ are radial covariance functions with radial components $c_1$ and $c_2$, respectively, as in \Cref{ass:gneiting}, that are regularly varying with index $-\rho_1<0$ and $-\rho_2$, respectively, then \cref{prop:UCTregular} implies that, as $t \to\infty$ (again, we write $t_1$ and $t_2$ omitting the dependence on $t$ for brevity), \begin{equation*} 
    t_1^{2d_1} t_2^{2d_2} c_2(t_2)^R c_1(t_1 c_2(t_2)^{1/d_1})^R \sim  t_1^{2d_1-R\rho_1}  t_2^{2d_2-R\rho_2(1-\frac{\rho_1}{d_1})} \underbrace{L_1(t_1c_2(t_2)^{1/d_1})^{R} L_2(t_2)^{R(1-\frac{\rho_1}{d_1})}}_{\text{slowly var.}}.
\end{equation*}
We remark that the slowly varying component grows slower than any polynomial.
After a change of variable, we can write \begin{align} 
    &  \frac{\Var(Y_{t_1,t_2})}{t_1^{2d_1} t_2^{2d_2} c_2(t_2)^R c_1(t_1 c_2(t_2)^{1/d_1})^R}  = R! \int_{(D_2)^2} \frac{c_2(t_2\|x_2-y_2\|)^R}{c_2(t_2)^R}   \times \nonumber \\
    & \quad \times \medmath{\left(\int_{(D_1)^2} \frac{c_1(t_1\|x_1-y_1\| c_2(t_2\|x_2-y_2\|)^{1/d_1})^R}{c_1(t_1 c_2(t_2\|x_2-y_2\|)^{1/d_1})^R}\frac{c_1(t_1 c_2(t_2\|x_2-y_2\|)^{1/d_1})^R}{c_1(t_1 c_2(t_2)^{1/d_1})^R} dx_1 dy_1 \right) dx_2 dy_2 }.
    \label{eq:var2}
\end{align}
We are going to prove that $\eqref{eq:var2}\to \ell^\prime \in (0,\infty)$ as $t \to\infty$, where \begin{equation*}
    \ell^\prime = R! \, \int_{(D_1)^2} \|x_1-y_1\|^{-\rho_1R}  dx_1 dy_1 \times \int_{(D_2)^2} \|x_2-y_2\|^{-\rho_2R(1-\rho_1/d_1)} dx_2 dy_2,
\end{equation*}
which is finite since $\rho_1R<d_1$ and $\rho_2R(1-\rho_1/d_1)<d_2$ by assumption. First, we prove pointwise convergence. For any $x_1,y_1\in D_1$ and $x_2,y_2\in D_2$, the regular variation property of $c_1$ and $c_2$ implies that (recall \Cref{ass:KEY}) \begin{align*}
   &   \lim_{t\to\infty} \frac{c_2(t_2\|x_2-y_2\|)^R}{c_2(t_2)^R} = \|x_2-y_2\|^{-R\rho_2},  \\
   &  \lim_{t\to\infty}  \frac{c_1(t_1\|x_1-y_1\| c_2(t_2\|x_2-y_2\|)^{1/d_1})^R}{c_1(t_1 c_2(t_2\|x_2-y_2\|)^{1/d_1})^R} = \|x_1-y_1\|^{-R\rho_1};
\end{align*}
applying, in addition, \cref{prop:UCTregular}, we also have that
\begin{align*}
      \lim_{t\to\infty} \frac{c_1(t_1 c_2(t_2\|x_2-y_2\|)^{1/d_1})^R}{c_1(t_1 c_2(t_2)^{1/d_1})^R}  & = \lim_{t\to\infty} \frac{c_1(t_1 c_2(t_2)^{1/d_1} (\|x_2-y_2\|^{-\rho_2/d_1}+o(t)))^R}{c_1(t_1 c_2(t_2)^{1/d_1})^R} \\
    & = \|x_2-y_2\|^{R\rho_1\rho_2/d_1}.
\end{align*}
To conclude, we will prove that the integrand function is bounded by an integrable one, see \eqref{eq:since}, and conclude by generalised dominated convergence theorem. The proof is divided into three main steps. In \textbf{I}, we define a partition of the integration domain into two subsets; then, in \textbf{II}, we show that in the first subset the covariance can be approximated by a power law; finally, in \textbf{III}, we show that the integral over the complementary is negligible.

\textbf{I. Partition.}
We fix constants $A$ and $\delta$, and by \cref{prop:potter}, we find two positive constants $X^{(1)}$ and $X^{(2)}$ such that Potter bounds holds for $c_i|_{[X^{(i)},+\infty)}$, $i=1,2$. Then, we define \begin{align*}
    & \mathcal D_{t_2} := \left\{ (x,y)\in(D_2)^2 : \|x-y\| \ge \frac{X^{(2)}}{t_2}\right\} , \\
    & \mathcal F_{t_1,t_2} : = \left\{ (x,y)\in (D_1)^2 :  \|x-y\| \ge \frac{X^{(1)}}{M^{1/d_1} t_1c_2(t_2)^{1/d_1}}  \right\}
\end{align*}
where $M := A^{-1} \diam(D_2)^{-\rho_2+\delta}$. Remark that as $t_1,t_2\to+\infty$, we have $\mathcal D_{t_2} \uparrow (D_2)^2$ and $\mathcal F_{t_1,t_2} \uparrow (D_1)^2$;
therefore, its complementary vanishes: $(D_1\times D_2)^2 \setminus \mathcal F_{t_1,t_2} \times \mathcal D_{t_2}\downarrow \emptyset$. Remark also that the complementary can be written as union of disjoint subsets, too: \begin{equation*}
    (D_1\times D_2)^2 \setminus \mathcal F_{t_1,t_2} \times \mathcal D_{t_2}  = \left( (D_1)^2\setminus \mathcal F_{t_1,t_2} \times \mathcal D_{t_2} \right) \cup \left( (D_1)^2 \times (D_2)^2\setminus \mathcal D_{t_2} \right).
\end{equation*}
Hence, \begin{align*}
    \eqref{eq:var2} = \iint_{\mathcal F_{t_1,t_2} \times \mathcal D_{t_2} } \dots \, + \iint_{(D_1)^2\setminus \mathcal F_{t_1,t_2} \times \mathcal D_{t_2}} \dots \, + \iint_{(D_1)^2 \times (D_2)^2\setminus \mathcal D_{t_2}}\dots \, =: I_1 + I_2 + I_3. 
\end{align*}

\textbf{II. Power law approximation.}  We show that we may choose $t_1$ and $t_2$ large enough so that, for every choice of $(x_1,y_1)\in\mathcal F_{t_1,t_2}$ and $(x_2,y_2)\in\mathcal D_{t_2}$, the following five inequalities hold: \begin{align*}
    & t_2  \ge X^{(2)}, 
    \qquad \qquad t_1 c_2(t_2)^{1/d_1} \ge X^{(2)}, \qquad \qquad 
    t_2 \|x-y\| \ge X^{(2)} , \\
    & t_1\|x_1-y_1\| c_2(t_2\|x_2-y_2\|)^{1/d_1} \ge X^{(1)},  
    \qquad \qquad t_1 c_2(t_2\|x_2-y_2\|)^{1/d_1} \ge X^{(1)},
\end{align*}
where $X^{(1)}$ and $X^{(2)}$ are constants that solely depend on $c_1$ and $c_2$.
If $\mathcal F_{t_1,t_2}$ and $\mathcal D_{t_2}$ are not empty, we can apply Potter bounds (\cref{prop:potter}) to trade $c_1$ and $c_2$ for power laws.

The first three are satisfied since we can choose $t_1$ and $t_2$ arbitrarily large, and thanks to \Cref{ass:KEY}. Now, we explain the fourth and fifth ones. We need to bound \begin{equation*}
    \inf_{x_2,y_2\in\mathcal D_{t_2}} c_2(t_2\|x_2-y_2\|).
\end{equation*}
Since \cref{prop:potter} holds for $c_2$ on the subset $\mathcal D_{t_2}$, we have that for any $\delta$, we find a constant $A$ such that \begin{align*}
    & \inf_{x_2,y_2\in\mathcal D_{t_2}} c_2(t_2\|x_2-y_2\|) \ge A^{-1} c_2(t_2) \inf_{x_2,y_2\in\mathcal D_{t_2}} \|x_2-y_2\|^{-\rho_2 \pm \delta}, \\
    & \inf_{x_2,y_2\in\mathcal D_{t_2}} c_2(t_2\|x_2-y_2\|) \le A c_2(t_2) \inf_{x_2,y_2\in\mathcal D_{t_2}} \|x_2-y_2\|^{-\rho_2 \pm \delta},
\end{align*}
Since $-\rho_2$ is negative, we may choose $\delta$ small enough so that $-\rho\pm\delta$ is negative. Then, \begin{equation*}
     c_2 (t_2\|x_2-y_2\|) \ge  c_2(t_2) A^{-1} \diam(D_2)^{-\rho_2 \pm \delta} \qquad \forall x_2,y_2\in D_2.
\end{equation*}
Therefore, thanks to \Cref{ass:KEY}, the fifth bound holds for $t_1$ large enough, and the fourth one holds for any $x_1,y_1\in \mathcal F_{t_1,t_2}$ by construction. Therefore, we can apply \cref{prop:potter} to any of the three fractions in \eqref{eq:var2}, and write, for $t_1,t_2$ large enough, the following upper bound: \begin{align}
    & I_1 = \int_{\mathcal D_{t_2}} \frac{c_2(t_2\|x_2-y_2\|)^R}{c_2(t_2)^R}   \times \nonumber\\
    & \times \medmath{\left(\int_{\mathcal F_{t_1,t_2}} \frac{c_1(t_1\|x_1-y_1\| c_2(t_2\|x_2-y_2\|)^{1/d_1})^R}{c_1(t_1 c_2(t_2\|x_2-y_2\|)^{1/d_1})^R}\frac{c_1(t_1 c_2(t_2\|x_2-y_2\|)^{1/d_1})^R}{c_1(t_1 c_2(t_2)^{1/d_1})^R} dx_1 dy_1 \right) dx_2 dy_2} \nonumber \\
    & \le A^{-3} \int_{\mathcal D_{t_2}} \|x_2-y_2\|^{R(-\rho_2\pm\delta)} \times  \nonumber \\
    & \qquad \times \left(\int_{\mathcal F_{t_1,t_2}} \|x_1-y_1\|^{R(-\rho_1\pm\delta)} \left(\frac{c_2(t_2\|x_2-y_2\|)^{R/d_1}}{c_2(t_2)^{R/d_1}} \right)^{-\rho_1\pm \delta} dx_1dy_1 \right) dx_2 dy_2 \nonumber \\
    & \le A^{-3}  \int_{\mathcal D_{t_2}} \|x_2-y_2\|^{R(-\rho_2\pm\delta)(1 + (-\rho_1\pm\delta)/d_1)}  dx_2 dy_2  \int_{\mathcal F_{t_1,t_2}} \|x_1-y_1\|^{R(-\rho_1\pm\delta)}  dx_1dy_1.\label{eq:since}
\end{align}
Since, by hypothesis, $ \rho_1 R <d_1$ and $\rho_2 > \frac{d_1d_2}{R(d_1-\rho_1)}$, we may choose $\delta$ small enough so that $R(\rho_1\pm\delta)< d_1 $, and that $R(-\rho_2\pm\delta)(1 + (-\rho_1\pm\delta)/d_1) > - d_2$, we have that \eqref{eq:since} is bounded.

\textbf{III. The reminder.} We shall prove that 
\begin{align*} 
    I_2& := \medmath{\int_{\mathcal D_{t_2}} \frac{c_2(t_2\|x_2-y_2\|)^R}{c_2(t_2)^R}    \left(\int_{(D_1)^2\setminus \mathcal F_{t_1,t_2}} \frac{c_1(t_1\|x_1-y_1\| c_2(t_2\|x_2-y_2\|)^{1/d_1})^R}{c_1(t_1 c_2(t_2)^{1/d_1})^R} dx_1 dy_1 \right) dx_2 dy_2,}\\
    I_3& := \medmath{\int_{(D_2)^2\setminus \mathcal D_{t_2}} \frac{c_2(t_2\|x_2-y_2\|)^R}{c_2(t_2)^R}   \times \left(\int_{(D_1)^2} \frac{c_1(t_1\|x_1-y_1\| c_2(t_2\|x_2-y_2\|)^{1/d_1})^R}{c_1(t_1 c_2(t_2)^{1/d_1})^R} dx_1 dy_1 \right) dx_2 dy_2,}
\end{align*}
both converge to $0$ as $t \to\infty$.
First term, $I_2$. Since we are integrating over $\mathcal D_{t_2}$, we can apply Potter bounds for the first fraction, involving $c_2$. Since the volume of the integration domain vanishes as $t_1\to\infty$, it's enough to bound $c_1\le 1$. Then, applying \cref{prop:potter} to $c_2$, we find $A>0$, such that for $t_2$ large enough
\begin{align*}
       & \int_{\mathcal D_{t_2}} \frac{c_2(t_2\|x_2-y_2\|)^R}{c_2(t_2)^R}    \left(\int_{(D_1)^2\setminus \mathcal F_{t_1,t_2}} \frac{c_1(t_1\|x_1-y_1\| c_2(t_2\|x_2-y_2\|)^{1/d_1})^R}{c_1(t_1 c_2(t_2)^{1/d_1})^R} \medmath{dx_1 dy_1} \right) \medmath{dx_2 dy_2} \\
       & \le A \underbrace{\int_{\mathcal D_{t_2}} \|x_2-y_2\|^{R(-\rho_2\pm\delta)} dx_2 dy_2}_{\text{bounded}} \times \frac1{c_1(t_1 c_2(t_2)^{1/d_1})^R} \vol((D_1)^2\setminus \mathcal F_{t_1,t_2}), 
       \end{align*}
       where the boundedness comes from dominated convergence. The domain $(D_1)^2\setminus \mathcal F_{t_1,t_2}$ is comparable to $D_1 \times $ a ball of radius $\asymp \frac1{t_1c_2(t_2)^{1/d_1}}$; hence, its volume is comparable to $ \frac1{t_1^{d_1} c_2(t_2)}$. 
       Therefore, up to some other constant $A>0$, recalling \Cref{ass:KEY} and since $c_1$ is asymptotically a power law of exponent $-\rho_1<0$ such that $R\rho_1<d_1$, we have that, as $t\to\infty$,
       \begin{align*}
       I_2 & \le  A \frac1{ \left( t_1c_2(t_2)^{1/d_1}\right)^{d_1-R\rho_1}} \longrightarrow 0.
\end{align*}
Second term, $I_3$. Remark that \begin{equation}
    c_2(X^{(2)}) \le \inf_{x_2,y_2\in (D_2)^2\setminus \mathcal D_{t_2}} c_2(t_2\|x_2-y_2\|) \le 1. \label{eq:nice lower}
\end{equation}
The right inequality is immediate, and the left one follows from the fact that $c_2$ attains its maximum on $\textrm{diag}((D_2)^2)$. Moreover,   $(D_2)^2\setminus \mathcal D_{t_2} \downarrow \textrm{diag}((D_2)^2)$. Hence, we may suppose that $\inf_{x_2,y_2\in (D_2)^2\setminus \mathcal D_{t_2}} c_2(t_2\|x_2-y_2\|)$ is always bigger than the value at $X^{(2)}$ (otherwise we may split the domain again, and the subset where the inequality fails decreases to $\emptyset$ as $t_2\to\infty$). Therefore, we may write \begin{align*}
    I_3 & =  \int_{(D_2)^2\setminus \mathcal D_{t_2}} \frac{c_2(t_2\|x_2-y_2\|)^R}{c_2(t_2)^R}   \times \Bigg(\int_{\widetilde{\mathcal F_{t_1}}} \frac{c_1(t_1\|x_1-y_1\| c_2(t_2\|x_2-y_2\|)^{1/d_1})^R}{c_1(t_1 c_2(t_2)^{1/d_1})^R} dx_1 dy_1   \\
    & \hspace{2cm} + \int_{(D_1)^2\setminus \widetilde {\mathcal F_{t_1}}} \frac{c_1(t_1\|x_1-y_1\| c_2(t_2\|x_2-y_2\|)^{1/d_1})^R}{c_1(t_1 c_2(t_2)^{1/d_1})^R} dx_1 dy_1 \Bigg) dx_2 dy_2 \\
    & = : I_4 + I_5,
\end{align*}
where $\widetilde{\mathcal F_{t_1}} := \{(x_1,y_1)\in (D_1)^2 : \|x_1-y_1\| \ge \frac{X^{(1)}}{t_1c_2(X^{(2)})^{1/d_1}}\}$. In the first addendum, \cref{prop:potter} (right inequality) applies to $c_1$: for some constant $A$ and $t_1$, $t_2$ large enough, we have \begin{multline}
   I_4 \le A  \int_{(D_2)^2\setminus \mathcal D_{t_2}} \left(\frac{c_2(t_2\|x_2-y_2\|)}{c_2(t_2)}\right)^{R(1+(-\rho_1\pm\delta)/d_1)}  dx_2dy_2 \times \\ \times \int_{\widetilde{\mathcal F_{t_1}}} \|x_1-y_1\|^{R(-\rho_1\pm\delta)} dx_1dy_1.
\end{multline}
Since $R\rho_1 < d_1$, we may choose $\delta$ small enough so that the second factor is bounded. Remark also that, as $t_1\to\infty$, $(D_1)^2\setminus \widetilde{  \mathcal F}_{t_1}\downarrow \emptyset$ comparable to a ball of radius $t_1^{-1}$ in $\R^{d_1}$. For the first factor, we bound the numerator $c_2\le 1$ and obtain, up to some other constant $A$, that \begin{align*}
    I_4 \le A  \frac1{c_2(t_2)^{R(1+(-\rho_1\pm\delta)/d_1)}} \vol( (D_2)^2\setminus \mathcal D_{t_2} ) \le A \frac1{t_2^{d_2-\rho_2R(1+(-\rho_1\pm\delta)/d_1)}} \longrightarrow 0 ,
\end{align*}
since the volume of $(D_2)^2\setminus \mathcal D_{t_2}$ decreases as the volume of a ball of radius $\frac1{t_2}$ in $\R^{d_2}$. Convergence to zero follows by taking $\delta$ small enough, since we are assuming that $\rho_2< \frac{d_1d_2}{R(d_1-\rho_1)}$. For the term $I_5$, we use the trivial bound $c_i\le 1$, for $i=1,2$ and compare the denominator with the growth of the volume of the integration domain. Recall that $c_i$ is asymptotically power law with exponent $-\rho_i<0$, $i=1,2$. Then, up to some constant $A$, 
\begin{align*}
    I_5 & =\int_{(D_2)^2\setminus \mathcal D_{t_2}} \frac{c_2(t_2\|x_2-y_2\|)^R}{c_2(t_2)^R}   \times \int_{(D_1)^2\setminus \widetilde{\mathcal F_{t_1}}} \frac{c_1(t_1\|x_1-y_1\| c_2(t_2\|x_2-y_2\|)^{1/d_1})^R}{c_1(t_1 c_2(t_2)^{1/d_1})^R} dx_1 dy_1 \\
    & \le A  \frac1{c_2(t_2)^{R}} \vol ((D_2)^2\setminus \mathcal D_{t_2}) \, \times \,\frac1{c_1(t_1 c_2(t_2)^{1/d_1})^{R}} \vol((D_1)^2\setminus \widetilde{\mathcal F_{t_1}}) \\
    & \le A \frac1{t_1^{d_1-R\rho_1}} \, \frac1{t_2^{d_2 -\rho_2R(1-\rho_1/d_1)}}  \longrightarrow 0.
\end{align*}
This allows us to conclude the proof.
\end{proof}

\begin{proof}[Sketch of the proof of \cref{prop:variance} \emph{(3)}]
   Recall that $C_1\notin L^R(\R^{d_1})$, with regular variation of parameter $-\rho_1$, and that $C_2\in L^{R(1-\rho_1/d_1)}$. Changing variables $x_1,y_1\to t_1x_1,t_1y_1$, and $x_2-y_2\to z_2$, we may write \begin{multline} 
   \label{eq:running}
       \frac{1}{t_2^{d_2}t_1^{2d_1}c_1(t_1)^R}  \int_{(t_1D_1\times t_2D_2)^2} c_2(\|x_2-y_2\|)^R c_1(\|x_1-y_1\| c_2(\|x_2-y_2\|)^{1/d_1})^R \, \medmath{dx_1dx_2dy_1dy_2} \\
       \\ = \int_{\R^{d_2}} c_2(\|z_2\|)^R g_2(z_2/t_2) \times  \int_{(D_1)^2} \frac{c_1(t_1\|x_1-y_1\| c_2(\|z_2\|)^{1/d_1})^R}{c_1(t_1)^R} dx_1dy_1 dz_2,
   \end{multline} 
   where $g_2(z) := g_2(D_2,z)$ denotes the covariogram of $D_2$, as defined in \cref{def:covariogram}. For fixed $x_i,y_i\in \R^{d_i}$, since $c_1$ is regularly varying with parameter $-\rho_1<0$, we have that, as $t \to\infty$, \begin{multline*}
       c_2(\|z_2\|)^R g_2(z_2/t_2) \times \frac{c_1(t_1\|x_1-y_1\| c_2(\|z_2\|)^{1/d_1})^R}{c_1(t_1)^R} \\ 
       \rightarrow c_2(\|z_2\|)^{R(1-\rho_1/d_1)} \vol(D_2) \times \|x_1-y_1\|^{-R\rho_1}.
   \end{multline*}
   With analogous (but more tedious) computations to the ones of the proof of case \emph{(4)}, one can also prove that a domination occurs. Therefore, as $t \to\infty$, we have \begin{equation*}
       \eqref{eq:running} \rightarrow \vol(D_2) \int_{\R^{d_2}} C_2(z_2)^{R(1-\rho_1/d_1)} \, dz_1 \times \int_{(D_1)^2} \|x_1-y_1\|^{-R\rho_1} \, dx_1dy_1 < \infty. 
   \end{equation*}
\end{proof}

\subsection{Proofs of limit theorems}
\label{subsec:proofLT}

In this Section we prove our main results, that are Theorems \ref{thm:ltBREUERMAJOR}, \ref{thm:ltBALTO}, and \ref{thm:ltROSENBLATT}. 
Preliminary to that, in \cref{subsec:reductionapply}, we show a reduction principle that will be used in the proof of each result. Next, in \cref{subsec:central}, we will prove Theorems \ref{thm:ltBREUERMAJOR} and \ref{thm:ltBALTO}. Finally, in \cref{subsec:proofNCLT}, we will prove \cref{thm:ltROSENBLATT}.

\textbf{Notation.} In the following, we will write $t_1$ and $t_2$ omitting the dependence on $t$. Recall that $t_1,t_2\to \infty$ as $t\to\infty$, as in \Cref{ass:KEY}.

\subsubsection{Dominating chaos}\label{subsec:reductionapply} 
In this subsection we show that, under suitable assumptions on $B$ (always satisfied in our setting), the following functionals share the same limiting distribution as $t \rightarrow\infty$
\begin{equation*}
\widetilde Y(t) = \widetilde{\left(\int_{t_1D_1\times t_2D_2} \f(B_x) dx\right) }, \quad \text{ and } \quad \widetilde Y_{t_1,t_2}[R]:=sgn(a_R)\widetilde{\left(\int_{t_1D_1\times t_2D_2} H_R(B_x) dx\right)}\,,
\end{equation*}
where $R$ Hermite rank of $\varphi$. This is not a novel result when $C^R\ge0$ (as in the Gneiting case) and some additional minimal assumptions on $C$ are satisfied, see e.g \cite{MN22,LRM23,pdom24,MainiRossiZheng2025}. When this assumptions are not satisfied, one can still hope in other techniques to control the first odd chaoses that avoid positivity assumptions on $C^R$, see e.g. \cite{GMT24,Maini2024, Gass2025}. 
What we need for our purposes is collected in the following proposition.

\begin{proposition}
\label{prop:reduction}
Let $\widetilde Y(t)$, $\widetilde Y(t)[R]$ be as above. Suppose that Assumptions \ref{ass:gneiting} and \ref{ass:KEY} hold. Suppose also that $C$ is as in \cref{prop:variance}, cases \emph{(2)}-\emph{(3)}-\emph{(4)}. Then, $\E [(\widetilde Y(t)-\widetilde Y(t)[R])^2]\to 0$, as $t\to\infty$. Therefore, they share the same limit in distribution.
\end{proposition}
\begin{proof}
    Since $C\ge0$ by Assumption \ref{ass:gneiting}, $C\le 1$ and $\|\varphi\|^2_{L^2(\gamma)}=\sum_{q=0}^\infty q!a_q^2<\infty$, we have \begin{align*}
        \frac{\Var(Y(t)-Y(t)[R])}{\Var(Y(t)[R])} & =\frac{\sum_{q\ge R+1}q!a_q^2 \int_{t_1D_1\times t_2D_2} C(x-y)^q \, dx \, dy}{R!a_R^2\int_{t_1D_1\times t_2D_2} C(x-y)^R \, dx \, dy} \\
        & \lesssim \, \frac{\int_{t_1D_1\times t_2D_2} C(x-y)^{R+1} \, dx\,dy}{\int_{t_1D_1\times t_2D_2} C(x-y)^{R} \, dx\,dy}\,.
    \end{align*}
    Since $\Var(Y(t)[R])\asymp\int_{t_1D_1\times t_2D_2} C(x-y)^{R} \, dx\,dy$, looking carefully at the rates of the variances in the statement of Proposition \ref{prop:variance}, we have that the last quotient always goes to $0$\footnote{Here the idea is that, even if $C\notin L^{R+1}$, the rate of $\int C^{R+1}$ is slower  if $C\notin L^R$. This is a standard behavior when $C^R\ge0$ and we are in the long memory regime, i.e. $C\notin L^R$.} in the cases \emph{(2)}-\emph{(3)}-\emph{(4)}. Thus, by standard techniques, see e.g. \cite[p.747]{MN22}, we conclude observing that, as $t\to\infty$, \begin{equation*}
        \frac{\Var(Y(t)-Y(t)[R])}{\Var(Y(t)[R])}\rightarrow 0 \quad\implies\quad \E[(\widetilde Y(t)-\widetilde Y(t)[R])^2]\rightarrow 0\,.
    \end{equation*}
\end{proof}

\subsubsection{Proofs of \cref{thm:ltBREUERMAJOR} and \cref{thm:ltBALTO}}
\label{subsec:central}

\begin{proof}[Proof of \cref{thm:ltBREUERMAJOR}]
    Thanks to \cref{eq: Y deco}, $\widetilde Y(t)$ admits the following chaotic decomposition: \begin{align*}
         Y(t) & = t_1^{d_1} t_2^{d_2} \vol(D_1\times D_2) \, \E[\varphi(B_0)] + \sum_{q=R}^\infty a_q \int_{t_1D_1\times t_2D_2} H_q(B_x) \, dx, 
    \end{align*}
    where $R$ is the Hermite rank of $\varphi$.
    Thanks to the orthogonality of any pair of chaotic components, and \cref{prop:variance} \emph{(1)}, we may write \begin{equation*}
        \Var(Y(t)) = 
        t_1^{d_1} t_2^{d_2} \, \vol(D_1)\vol(D_2) \sum_{q=R}^\infty a_q^2 \, q! \, \|C_1\|^q_{L^q(\R^{d_1})} \|C_2\|^q_{L^{q-1}(\R^{d_2})}.
    \end{equation*}
    Therefore, recalling \cref{eq: Y deco}, we have \begin{equation*}
        \widetilde Y(t) = \sum_{q=R}^\infty I_q\left(\frac{a_q f_{t_1D_2\times t_2D_2,q}}{\sqrt{\Var(Y(t))}}\right) =: \sum_{q=R}^\infty I_q (\widetilde f_{t_1,t_2,q}).
    \end{equation*} 
    By means of \cref{thm 631}, it's enough to check four conditions: \begin{enumerate}
        \item for every fixed $q\ge R$, $q!\|\widetilde f_{t_1,t_2,q}\|\to \sigma_q^2$ as $t_1,t_2\to\infty$; 
        \item $\sigma^2 = \sum_{q=R}^\infty \sigma_q^2 <\infty$;
        \item for all $q\ge R$, and $r=1,\ldots, q-1$, $\|\widetilde f_{t_1,t_2,q}\otimes_r \widetilde f_{t_1,t_2,q}\|\to 0$;
        \item $\lim_{N\to\infty} \sup_{t_1,t_2} \sum_{q=N+1}^\infty q! \|\widetilde f_{t_1,t_2,q}\|^2 \to 0$.
    \end{enumerate}
    Since \begin{equation*}
        \| f_{t_1,t_2,q}\|^2 = \int_{t_1D_1\times t_2D_2} C(x-y)^q \, dx \, dy, 
    \end{equation*}
    thanks to \cref{prop:variance}, we have  that the first condition is satisfied with $\sigma_q^2 = q! a_q^2$, and the second one, too, with $\sigma^2 = 1$. The fourth condition holds as well, since the series is convergent. Now, we prove that the third is satisfied. We may write, up to some constant $K$, \begin{align*}
        \|\widetilde f_{t_1,t_2,q} \otimes_r  \widetilde f_{t_1,t_2,q}\|^2 = K \medmath{\frac{\int_{(t_1D_1\times t_2D_2)^4} C(x-y)^{q-r}C(y-z)^rC(z-v)^{q-r}C(v-x)^r \, dx\,dy\,dz\,dv}{t_1^{2d_1}t_2^{2d_2}}}.
    \end{align*} 
    Using the classical inequality $A^{q-r}B^r\le K( A^q+B^q)$ for any $A,B\in\R^+$, and a positive constant $K$, changing variables $\hat x = x-y$, we get: \begin{align}
    \label{eq:integrable hence BMalpha}
        & \int_{(t_1D_1\times t_2D_2)^4} C(x-y)^{q-r}C(y-z)^rC(z-v)^{q-r}C(v-x)^r \, dx\,dy\,dz\,dv \\
        & \le 2 K  \int_{\R^d} C(\hat x)^q 1_{(t_1D_1\times t_2D_2)-y}(\hat x) \int_{(t_1D_1\times t_2D_2)^3} C(y-z)^{q-r}C(z-v)^r \, dy \, dz \, dv \, d\hat x \nonumber \\
        & \le 2K \underbrace{\int_{\R^d} C(\hat x)^q d\hat x}_{\text{bounded}} \int_{(t_1D_1\times t_2D_2)^3} C(y-z)^{q-r}C(z-v)^r \, dy \, dz \, dv,\label{eq:mid}
    \end{align}
    since $q\ge R$, $C\in L^R{\R^d}$ and $C\le 1$. Changing variables as $a=y-z$, $b=z-v$, and $c=v$, we get (up to some other constant $K$) \begin{align}
        & \eqref{eq:mid}  \le  K \int_{\R^{2d}} C(a)^{q-r}C(b)^r \times \nonumber\\
        & \qquad \times \int_{\R} 1_{t_1D_1\times t_2D_2} (a+b+c) 1_{t_1D_1\times t_2D_2} (b+c) 1_{t_1D_1\times t_2D_2} (c)  \, dc \, db \, da \nonumber\\
        & \le K \int_{\R^{2d}} C(a)^{q-r}C(b)^r \times \underbrace{\sqrt{\int_{\R^{d_1}} 1_{t_1D_1\times t_2D_2} (a+b+c) 1_{t_1D_1\times t_2D_2} (b+c)   dc}}_{=: f(a)} \nonumber\\
        & \hspace{4cm}\times  \sqrt{\int_{\R^{d_1}} 1_{t_1D_1\times t_2D_2} (b+c) 1_{t_1D_1\times t_2D_2} (c)  dc} \, da db ,
        \label{eq:integrable hence BMomega}
    \end{align}
    by applying H\"older inequality in the integral $\int_{\R^d} \ldots dc$. Remark that the function $f:y\in\R^{d}\mapsto  \sqrt{\vol(t_1D_1\times t_2D_2\cap (t_1D_1\times t_2D_2-y))}$ has compact support contained in the product \begin{equation*}
        B(t_1,t_2):=\mathcal B_{d_1}(t_1\diam(D_1))\times \mathcal B_{d_2}(t_2\diam(D_2)),
    \end{equation*}
    where $\mathcal B_{d}(s)$ denotes the unit ball in $\R^d$ of radius $s>0$. We can write, for all $y\in\R^d$, $f(y)\le 1_{B(t_1,t_2)}(y) \sqrt{\vol(t_1D_1\times t_2D_2)}$. 
    Therefore, \begin{align*}
        & \|\widetilde f_{t_1,t_2,q}\otimes_r \widetilde f_{t_1,t_2,q}\|^2  \le K \, t_1^{-2d_1} t_2^{-2d_2} \left( \int_{\R^d} C(a)^{q-r}f(a) \, da \right)\, \left( \int_{\R^d} C(b)^{r}f(b) \, db \right)\\
        & \le K t_1^{-d_1+d_1\frac{q-r}{q}}t_2^{-d_2+d_2\frac{q-r}{q}} \left( \int_{B(t_1,t_2)} C(a)^{q-r} \, da \right)\, t_1^{-d_1+d_1\frac{r}{q}}t_2^{-d_2+d_2\frac{r}{q}} \left( \int_{B(t_1,t_2)} C(a)^{r} \, da \right).
    \end{align*}
    To prove the desired convergence, it remains to show that for, any $r=1,\ldots,q-1$, \begin{equation}
    \label{eq:integrable hence 0}
        t_1^{d_1(-1+\frac{q-r}{q})}t_2^{d_2(-1+\frac{q-r}{q})} \left( \int_{B(t_1,t_2)} C(a)^{q-r} \, da \right) \longrightarrow 0.
    \end{equation}
    This can be proved as in the proof of \cite[Theorem 7.2.4]{bluebook}, point (c) Equation (7.2.7), since $C\in L^q(\R^d)$ for all $q\ge R$. 

    Suppose now that $C_2\notin L^{R-1}(\R^{d_2})$, but it is slowly or regularly varying (in the latter case, with index $-\rho_2<0$ such that $\rho_2\le d_2/(R-1)$). By \cref{prop:reduction}, we have that $\widetilde Y(t)$ and $\widetilde Y(t)[R]$ share the same limiting distribution.  Since $\widetilde Y(t)[R]\subset \mathcal H_R$, see \cref{sec:malliavinstein}, we show that, for any $r=1,\ldots,R-1$, we have
    \begin{equation}
    \label{c4}
     \frac{\int_{(t_1D_1 \times t_2D_2)^4} C(x-y)^{R-r}C(y-z)^r C(z-v)^{R-r}  C(v-x)^r \, dx \, dy \, dz \, dv}{\left(\int_{(t_1D_1 \times t_2D_2)^2} C(x-y)^R\,  dx \, dy\right)^2}\longrightarrow 0 ,
\end{equation}
and conclude by applying \cref{thm:4th}.
The rate of growth of the denominator as $t \to\infty$ has already been studied in \cref{prop:variance} (see also \cref{rem:useful var C1 int}):  \begin{equation}
\label{var rates}
    \Var(Y(t)[R]) = R! \int_{(t_1D_1 \times t_2D_2)^2} C(x-y)^R\,  dx \, dy  \asymp \ell^\prime  \, t_1^{d_1} \, \int_{(t_2D_2)^2} C_2(x_2-y_2)^{R-1} dx_2 dy_2,
\end{equation}
where $\ell^\prime = R! \vol(D_1) \int_{\R^{d_1}} C_1(z_1)^R dz_1$. Recall \Cref{ass:gneiting}. Since $c_1$ and $c_2$ are decreasing\footnote{$c_1$ is always decreasing, being completely monotone. Since $1/c_2$ has a completely monotone derivative, we know that $(1/c_2)'$ is positive, implying that $c_2$ is decreasing.}, then for all $x_2,y_2\in t_2D_2$, we have the uniform estimate \begin{equation*}
    C_1\left((x_1-y_1)C_2(x_2-y_2)^{1/d_1}\right) \le C_1\left((x_1-y_1)c_2(t_2M)^{1/d_1}\right),
\end{equation*}
where $M:=\diam(D_2)$. 
Then, \begin{align}
    & \int_{(t_1D_1\times  t_2D_2)^4} C^{R-r}(x-y) C^r (y-z) C^{R-r}(z-v) C^r(v-x) \, dx \, dy \, dz \, dv  \nonumber \\ 
    & \le 
    \int_{(t_2D_2)^4} C_2^{R-r}(x_2-y_2) \ldots C_2^r(v_2-x_2) \, dx_2 \, dy_2 \, dz_2 \, dv_2 \times \label{I2}\\
    & \quad \times \underbrace{\int_{(t_1D_1)^4} C_1^{R-r}\left((x_1-y_1)c_2(t_2M)^{1/d_1}\right) \ldots C_1^r\left((v_1-x_1)c_2(t_2M)^{1/d_1}\right) \, dx_1 \, dy_1 \, dz_1 \, dv_1}_{=: J_1(t_1,t_2)} \label{I1}.
\end{align}
Since $C_1\in L^R(\R^{d_1})$, analogously to what we have done between \cref{eq:integrable hence BMalpha,eq:integrable hence BMomega}, we may prove that 
\begin{align} 
    J_1(t_1,t_2)  \le K \, \int_{\R^{d_1}} C_1\left(x_1 c_2(t_2M)^{1/d_1}\right)^R dx_1 & \times \left( \int_{\R^{d_1}} C_1\left(y_1 c_2(t_2M)^{1/d_1}\right)^{R-r} f(y_1) dy_1 \right) \nonumber \\
    &  \times \left( \int_{\R^{d_1}} C_1\left(z_1 c_2(t_2M)^{1/d_1}\right)^{r} f(z_1) dz_1 \right), \label{eq:daje}
\end{align}
where the function $f:y_1\in\R^{d_1}\mapsto  \sqrt{\vol(t_1D_1\cap (t_1D_1-y_1))}$ has support contained in a ball of radius $\diam(t_1D_1)$, denoted $t_1\mathcal B$, hence compact. Moreover, for all $y_1\in\R^{d_1}$ we have $f(y_1)\le 1_{t_1\mathcal B}(y_1)\sqrt{\vol(t_1D_1)}$. Then, changing variables, we get (for some $K$ that may change at every line) \begin{align*}
   J_1(t_1,t_2) & \le K  \times c_2(t_2M)^{-3} \underbrace{\int_{\R^{d_1}} C_1(x_1)^R dx_1}_{<\infty, \, C_1\in L^R(\R^{d_1})} \bigg( \int_{\left(t_1c_2(t_2M)^{1/d_1}\right)\mathcal B} C_1(y_1)^r dy_1 \times \sqrt{\vol(t_1D_1)} \bigg) \\
   & \hspace{3cm}\times \bigg( \int_{\left(t_1c_2(t_2M)^{1/d_1}\right)\mathcal B} C_1(y_1)^{R-r} dy_1 \times \sqrt{\vol(t_1D_1)} \bigg) \\
    & \le K \times  c_2(t_2M)^{-3} \times t_1^{d_1} \prod_{j=r,R-2}\bigg( \int_{\left(t_1c_2(t_2M)^{1/d_1}\right)\mathcal B} C_1(y_1)^{j} \, dy_1 \bigg) .  
\end{align*}
Recalling \eqref{var rates}, we have found that, up to some constant $K$, as $t\to\infty$, \begin{align*}
   & \eqref{c4}  \le K \, \times \underbrace{\left( \prod_{j=R,R-r}\frac{ \int_{\left(t_1c_2(t_2M)^{1/d_1}\right)\mathcal B} C_1(y_1)^j dy_1}{(t_1c_2(t_2M)^{1/d_1})^{(d_1 \, j)/R}}\right)^2}_{=: I_1(t_1,t_2)^2} \times c_2(t_2M)^{-2} \\
   & \times {\underbrace{\frac{ \int_{(t_2D_2)^4} C_2(x_2-y_2)^{R-r} C_2(y_2-z_2)^{r} C_2(z_2-v_2)^{R-r} C_2(v_2-x_2)^r \, dx_2 \, dy_2 \, dz_2 \, dv_2}{ \left(\int_{(t_2D_2)^2} C_2(x_2-y_2)^{R-1} \, dx_2\, dy_2\right)^2}}_{=:I_2(t_2)} }.
\end{align*} 
Since $C_2$ is slowly or regularly varying such that $C_2\notin L^{R-1}(\R^{d_2})$, $c_2(t_2M)^{-2} I_2(t_2)$ is bounded. To conclude, we may prove that $I_1(t_1,t_2)\to 0$ as $t\to\infty$, as done in the proof of \cite[Theorem 7.2.4]{bluebook}, point (c) Equation (7.2.7), since $C_1\in L^R(\R^d)$, applying also \cref{eq:integrable hence 0} and recalling  \Cref{ass:KEY}. Thanks to \cref{thm:4th}, we conclude. 
\end{proof}

\begin{proof}[Proof of \cref{thm:ltBALTO}] Recall that $C_2\in L^{R(1-\rho_1/d_1)}(\R^{d_1})$ (with no further assumptions on its regular variation) and suppose $C_1\notin L^R(\R^{d_1})$ is regularly varying with parameter $-\rho_1<0 $ such that $ \rho_1 < d_2/R$. By \cref{prop:reduction}, $\widetilde Y(t)$ and $\widetilde Y(t)[R]$ share the same limiting distribution. Since $\widetilde Y(t)[R]\subset \mathcal H_R$, see \cref{sec:malliavinstein}, we show that, for any $r=1,\ldots,R-1$, we have
    \begin{equation}
    \label{c43}
     \frac{\int_{(t_1D_1 \times t_2D_2)^4} C(x-y)^{R-r}C(y-z)^r C(z-v)^{R-r}  C(v-x)^r \, dx \, dy \, dz \, dv}{\left(\int_{(t_1D_1 \times t_2D_2)^2} C(x-y)^R\,  dx \, dy\right)^2}\longrightarrow 0 ,
\end{equation}
and conclude by applying \cref{thm:4th}. 
Changing variables $a_i\mapsto t_ia_i$, for $i=1,2$ and $a\in \{x,y,z,v\}$, we may write the numerator of \eqref{c43} as \begin{align*}
    &  t_2^{4d_2} \int_{D_2^4} C_2(t_2(x_2-y_2))^{R-r} \ldots C_2(t_2(v_2-x_2))^r \, dx_2 \, dy_2 \, dz_2 \, dv_2 \times \\
    & \, \times t_1^{4d_1} \int_{D_1^4} C_1\left(t_1(x_1-y_1)C_2(t_2(x_2-y_2))^{1/d_1}\right)^{R-r} \ldots \\
    & \hspace{5cm} \ldots C_1\left(t_1(v_1-x_1)C_2(t_2(v_2-x_2))^{1/d_1}\right)^r dx_1 \ldots dv_1.
\end{align*}
Thanks to \cref{prop:variance}, and regular variation of $C_1$, we have that the denominator of \eqref{c43} is, up to some positive constant $K$, that may vary from line to line, \begin{equation*}
    \left( \int_{(t_1D_1\times t_2D_2)^2} C(x-y)^R \, dx \, dy \right)^2 \sim K t_1^{4d_1 - 2R\rho_1}  t_2^{2d_2} \sim K t_1^{4d_1} c_1(t_1)^{2R} \, t_2^{2d_2}.
\end{equation*}
Therefore, as $t\to\infty$, we have \begin{align*}
    \eqref{c43} & \le K \, t_2^{2d_2} \int_{D_2^4} C_2(t_2(x_2-y_2))^{R-r} \ldots C_2(t_2(v_2-x_2))^r \, dx_2 \, dy_2 \, dz_2 \, dv_2 \times \\
    & \qquad \times \frac1{c_1(t_1)^{2R}} \int_{D_1^4} C_1\left(t_1(x_1-y_1)C_2(t_2(x_2-y_2))^{1/d_1}\right)^{R-r} \ldots  dx_1 \ldots dv_1.
\end{align*}
We define the subsets $\widehat{\mathcal D}_{t_2}$ and $\widehat {\mathcal F}_{t_1,t_2}$ in such a way that we may apply \cref{prop:potter} to $C_1$ (we leave the details, but the definitions are analogous to the ones given in the proof of \cref{prop:variance}). Therefore, up to some constant $K$, for $t_1$ and $t_2$ large enough, and $\delta>0$, we have that \begin{align}
&  \eqref{c43} 
      \le  K \times  t_2^{2d_2} \times \nonumber\\
     &  \times \int_{\widehat D_{t_2}} C_2(t_2(x_2-y_2))^{(R-r)(1-(\rho_1+\delta)/d_1)} \ldots C_2(t_2(v_2-x_2))^{r(1-(\rho_1+\delta)/d_1)} \, \medmath{dx_2 \, dy_2 \, dz_2 \, dv_2 } \nonumber\\
    &  \times \int_{\widehat F_{t_1,t_2}} \|x_1-y_1 \|^{-{(R-r)}(\rho_1+\delta)} \ldots \|v_1-x_1 \|^{-r(\rho_1+\delta)} \  \medmath{dx_1 \, dy_1 \, dz_1 \, dv_1} \nonumber\\
    & + \text{reminder}.\label{I2 di nuovo}
\end{align}
Since we may choose $\delta$ small enough so that $C_2^{1+(-\rho_1+\delta)/d_1}\in L^R(\R^{d_1})$, by an application of the same arguments from \cref{eq:integrable hence BMalpha,eq:integrable hence BMomega}, we find that $\eqref{I2 di nuovo}\to 0$ as $t\to \infty$. To conclude, we remark that the second line is bounded since $\rho_1<d_1/R$. Finally, the reminder is negligible, since the volumes of $(D_1)^4\setminus \widehat{\mathcal F}_{t_1,t_2}$ and $(D_2)^4\setminus \widehat{\mathcal D}_{t_2}$ vanish like volumes of balls in $\R^{d_1}$ and $\R^{d_2}$ of radius $\frac1{t_1c_2(t_2)^{1/d_1}}$ and $\frac1{t_2}$, respectively.
\end{proof}

\subsubsection{Proof of \cref{thm:ltROSENBLATT}}

\label{subsec:proofNCLT}

\begin{proof}[Proof of \cref{thm:ltROSENBLATT}]
By \cref{prop:reduction}, $\widetilde Y(t)$ and $\widetilde Y(t)[2]$ share the same limiting distribution. We show that we have convergence of the cumulants of $\widetilde Y(t)[2]$ to the ones of a $2$-domain Rosenblatt distribution, explicitly computed in \cref{prop:rosencumulants}. This is enough to conclude since the Rosenblatt random variables and $\widetilde Y(t)[2]$ belongs to the second chaos, hence their distributions are moment-determined, see \cref{fredholm does the job}. Since moments and cumulants determine each other uniquely, we can equivalently study the cumulants of $\widetilde Y(t)$, denoted by $\kappa_k(t)$, for $k\ge 1$. Since $\widetilde Y(t)\in \mathcal H_2$, we have that $\kappa_1(t) = 0$, and standardization implies $\kappa_2(t)=1$ for any $t>0$. For higher $k\ge 2$, we have that \begin{equation*}
    \kappa_k(t) = 2^{k-1}(k-1)! \, c_k(t),
\end{equation*}
where \begin{equation*}
    c_k(t) := \frac{\int_{(t_1D_1 \times t_2D_2)^k} C(x^{(1)}-x^{(2)})C(x^{(2)}-x^{(3)}) \cdots C(x^{(k)}-x^{(1)}) dx^{(1)}  \ldots dx^{(k)}}{\left(2\int_{(t_1D_1 \times t_2D_2)^2} C^2(x^{(1)}-x^{(2)}) dx^{(1)}dx^{(2)}\right)^{k/2}},
\end{equation*}
and $C$ is the covariance function of the underlying Gaussian field $B$. The numerator of $c_k(t)$ may be written in terms of a cyclic product as in \cref{cyclic difficile}: \begin{equation*}
    \int_{(t_1D_1 \times t_2D_2)^k} C(x^{(1)}-x^{(2)})  \cdots C(x^{(k)}-x^{(1)}) dx^{(1)}  \ldots dx^{(k)} = \|C^{\circ k}\,\|_{L^1((t_1D_1\times t_2D_2)^k)}.
\end{equation*}
Thank to \cref{prop:variance}\emph{(4)}, the denominator is asymptotically equivalent, as $t\to\infty$, to \begin{equation*}
\medmath{    \left(2\int_{(t_1D_1 \times t_2D_2)^2} C(x^{(1)}-x^{(2)})^2 dx^{(1)}dx^{(2)}\right)^{k/2} \sim (2\sigma^2)^{k/2} \left[t_1^{d_1}\,c_1\left(t_1c_2(t_2)^{1/d_1}\right)\right]^{k}\,\left[t_2^{d_2}\,c_2(t_2)\right]^{k} },
\end{equation*}
where $\sigma^2 = \int_{(D_1^2)} \|x_1-y_1\|^{-\rho_1} dx_1\, dy_1 \times \int_{(D_2^2)} \|x_2-y_2\|^{-\rho_2(1-\rho_1/d_1)} dx_2\, dy_2 $.
Thanks to \cref{lem:cactus gneiting nonclt}, we conclude that \begin{equation*}
    \lim_{t\to\infty} c_k(t) = c_k^{\rho_1,\rho_2(1-\rho_1/d_1)},
\end{equation*}
as defined in \cref{eq:ckRosen}. We conclude thanks to \cref{prop:rosenblatt characteristic}.
\end{proof}

\appendix
    \section{Some results on regularly varying functions over growing domains}
\label{sec:auxiliaryRV}

This Appendix is devoted to the study of the limit of (a proper renormalization of) the following: \begin{align}
    & \int_{(t\Delta)^k} c(\|x^{(1)}-x^{(2)}\|)\cdot \ldots \cdot c(\|x^{(k)}-x^{(1)}\|) \, dx^{(1)}\cdot \ldots \cdot dx^{(k)}, \qquad t\to\infty\label{cyclic facile}\\
    & \int_{(t_1D_1\times t_2D_2)^k} c(\|x^{(1)}-x^{(2)}\|)\cdot \ldots \cdot c(\|x^{(k)}-x^{(1)}\|) \, dx^{(1)}\cdot \ldots \cdot dx^{(k)},\qquad t_1,t_2\to\infty.\label{cyclic difficile}
\end{align}
where $c$ is a regularly varying function (see \cref{lem:cactus}), or a nested regularly varying function (see \cref{lem:cactus gneiting nonclt}). These results will play a key role in determining the asymptotic distribution of the normalization of \cref{eq:Yt}, under \Cref{ass:gneiting}, in the regularly varying long-range dependent case.

\medskip

Let us recall the definition of cyclic product, given in \cref{def:cyclic}. 
\begin{remark} We may extend the previous definition to the product of $k$ positive functions as follows.
    If $f_1,\dots,f_k\in L^1(\R^d)$ and $f_i\ge0$, then the cyclic product is defined as \begin{equation*}
        f_1\circ\cdots\circ f_k(x^{(1)},\dots,x^{(k)}):=\prod_{i=1}^k f_i\left(x^{(i)}-x^{(i+1)}\right)
    \end{equation*}
and  for $f_i$ symmetric ($f_i(x)=f_i(-x)$) we have by Young's convolution inequalities
\begin{equation}
\|f_1\circ\cdots\circ f_k\|_{L^1(\R^{dk})}
=
(f_1*\cdots*f_k)(0)
\le\|f_1\|_{L^2(\R^d)}\|f_2\|_{L^2(\R^d)}
\prod_{i=3}^k \|f_i\|_{L^1(\R^d)}\,\,.
\label{eq:young}
\end{equation}
In particular, if $\Delta\subset\R^d$ is compact we have (since $x^{(i)}-x^{(i+1)}\in\Delta-\Delta$ and by Cauchy-Schwarz)
\begin{multline}
    \label{eq:locallyint cyclic}
    \|f_1\circ\cdots\circ f_k\|_{L^1(\Delta^k)}\le \|f_1\|_{L^2(\Delta-\Delta)}\|f_2\|_{L^2(\Delta-\Delta)}
\prod_{i=3}^k \|f_i\|_{L^1(\Delta-\Delta)} \\
\le {\rm Vol}(\Delta-\Delta)\prod_{i=1}^k \|f_i\|_{L^2(\Delta-\Delta)}\,.
\end{multline}
\end{remark}

Before moving to the proofs, note that \eqref{cyclic facile} and \eqref{cyclic difficile} can be rewritten as follows \begin{equation*}
    \eqref{cyclic facile}=\|C^{\circ k}\|_{L^1((t\Delta)^k)}\quad\quad\quad \eqref{cyclic difficile}=\|C^{\circ k}\,\|_{L^1((t_1D_1\times t_2D_2)^k)}\,\,.
\end{equation*}
We start studying the asymptotic behavior of \eqref{cyclic facile}.
\begin{proposition}\label{lem:cactus}
Let $d\ge1$ and let $\Delta\subset\R^d$ be compact.
Assume $C(x)=c(\|x\|)$ where $c$ is regularly varying at infinity with index $-\rho$, with
$0<\rho<\frac d2$. Set $f_\rho(x):=\|x\|^{-\rho}$.
Then for $k\ge2$,
\begin{equation}
\label{eq:claim facile}
    \lim_{t\to\infty}
\frac{\|C^{\circ k}\|_{L^1((t\Delta)^k)}}{(t^d c(t))^k}
=\|f_\rho^{\circ k}\|_{L^1(\Delta^k)}\,\,
\end{equation}
where the limit is finite because $f^2_\rho$ is locally integrable, see \eqref{eq:locallyint cyclic}.
\end{proposition}
\begin{proof}
By the standard change of variable $x^{(i)}=t\,y^{(i)}$, we have \begin{equation*}
    \frac{\|C^{\circ k}\|_{L^1((t\Delta)^k)}}{(t^d c(t))^k}
=
\left\|
\left(\frac{C(t\cdot)}{c(t)}\right)^{\circ k}
\right\|_{L^1(\Delta^k)}.
\end{equation*}
By assumption, for every $x\neq 0$, $C(tx)/c(t)\rightarrow f_\rho(x)$.
Fix $X>0$ and define \begin{equation*}
    \mathcal{D}_{t}=\{x\in \Delta-\Delta:\,\|x\|\ge X/t\}.
\end{equation*}
By \cref{prop:potter}, $X$ can be chosen so that $C(t\cdot)/c(t)\mathbbm{1}_{\mathcal{D}_t}\lesssim
\max\{f_{\rho+\delta},f_{\rho-\delta}\}$,
that (since $\rho<d/2$) for some $\delta>0$ is locally in $L^2$.
Hence, by \eqref{eq:locallyint cyclic} and dominated convergence we get
\begin{equation}
\label{eq:dominating term facile}
\lim_{t\to\infty}
\left\|
\left(\frac{C(t\cdot)}{c(t)}\mathbbm{1}_{\mathcal{D}_t}\right)^{\circ k}
\right\|_{L^1(\Delta^k)}
=
\|f_\rho^{\circ k}\|_{L^1(\Delta^k)}.
\end{equation}
It remains to control the contribute on $\mathcal{D}_t^c$. We can decompose \begin{equation*}
    \left\|
\left(\frac{C(t\cdot)}{c(t)}\right)^{\circ k}
\right\|_{L^1(\Delta^k)}
=
\left\|
\left(\frac{C(t\cdot)}{c(t)}\mathbbm 1_{\mathcal D_t}\right)^{\circ k}
\right\|_{L^1(\Delta^k)}
+
r_t\quad \  \,.
\end{equation*}
By multi-linearity of the cyclic product, positivity of $C$ and applying \eqref{eq:locallyint cyclic}, we have \begin{equation*}
    r_t
\le
k\,
\left\|
\left(\frac{C(t\cdot)}{c(t)}\mathbbm 1_{\mathcal D_t^c}\right)
\circ
\left(\frac{C(t\cdot)}{c(t)}\right)^{\circ(k-1)}
\right\|_{L^1(\Delta^k)}\lesssim\,\left\|
\frac{C(t\cdot)}{c(t)}\mathbbm 1_{\mathcal D_t^c}\right\|_{L^1(\Delta-\Delta)}
\,
\left\|
\frac{C(t\cdot)}{c(t)}
\right\|^{k-1}_{L^2(\Delta-\Delta)}\,.
\end{equation*}
Since $|C|\le 1$ and \eqref{eq:regvar 0 e infty} holds for $c^2$ regularly varying with parameter $2\rho<d$, we have \begin{equation*}
    \left\|
\frac{C(t\cdot)}{c(t)}\mathbbm 1_{\mathcal D_t^c}
\right\|^2_{L^1(\Delta-\Delta)}\le \int_{\Delta-\Delta}\frac{\mathbbm{1}_{\|x\|\le X/t}}{c^2(t)}\,dx
\lesssim
\frac{1}{t^d c^2(t)}\rightarrow 0\,\,.
\end{equation*}
Moreover, using again \cref{prop:potter} on $\mathcal{D}_t$ with $\delta>0$ sufficiently small, we get \begin{equation*}
    \left\|
\frac{C(t\cdot)}{c(t)}
\right\|_{L^2(\Delta-\Delta)}\le\left\|
\frac{C(t\cdot)}{c(t)}\mathbbm 1_{\mathcal D_t}
\right\|_{L^2(\Delta-\Delta)}+\left\|
\frac{C(t\cdot)}{c(t)}\mathbbm 1_{\mathcal D_t^c}
\right\|_{L^2(\Delta-\Delta)}\lesssim \left\|
f_{\rho+\delta}\mathbbm 1_{\mathcal D_t}
\right\|_{L^2(\Delta-\Delta)}<\infty\,.
\end{equation*}
Hence, $r_t\rightarrow0$ and the proof is concluded. 
\end{proof}

\begin{proposition}
\label{lem:cactus gneiting nonclt}
    Fix $d_i\ge 1$ integer and $D_i\subseteq \R^{d_i}$ compact, $i=1,2$. Suppose that \Cref{ass:gneiting} and   \Cref{ass:KEY} hold, and the setting of \cref{thm:ltROSENBLATT} prevails, that is, 
    $$\rho_1<d_1/2\,,\quad\quad \quad\rho_2^{*}:=\left(1-\frac {\rho_1}{d_1}\right)\rho_2<\frac{d_2}{2}\,\,.$$
    Finally, set $f_{\rho_i}(x_i)=\|x_i\|^{-\rho_i}$.
    Then, for $k\ge 2$ we have as $t_1,t_2\rightarrow\infty$ (recall \Cref{ass:KEY})
    \begin{equation*}
\frac{\|C^{\circ k}\,\|_{L^1((t_1D_1\times t_2D_2)^k)}}{\left[t_1^{d_1}\,c_1\left(t_1c_2(t_2)^{1/d_1}\right)\right]^{k}\,\left[t_2^{d_2}\,c_2(t_2)\right]^{k}}  \longrightarrow\left\|f_{\rho_1}^{\circ k}\,\right\|_{L^1(D_1^k)}\left\|f_{\rho_2^*}^{\circ k}\,\right\|_{L^1(D_2^k)}\,.
    \end{equation*}
\end{proposition}

\begin{proof}
By \Cref{ass:gneiting}, standard changes of variables $x_i=t_i x_i'$ we can write \begin{equation*}
    \frac{\|C^{\circ k}\,\|_{L^1((t_1D_1\times t_2D_2)^k)}}{\left[t_1^{d_1}\,c_1\left(t_1c_2(t_2)^{1/d_1}\right)\right]^{k}\,\left[t_2^{d_2}\,c_2(t_2)\right]^{k}}=\left\|K^{\circ k}_{t_1,t_2}\right\|_{L^1(D^k)}\,.
\end{equation*}
where $D:=D_1\times D_2$ and (by \Cref{ass:RV} and \cref{prop:UCTregular}) for $x_1,x_2\neq 0$ we have \begin{equation*}
    K_{t_1,t_2}(x_1,x_2):=
\frac{c_1\big(t_1\,x_1\, c_2(t_2 x_2)^{1/d_1}\big)\,c_2(t_2 x_2)}
{c_1\big(t_1c_2(t_2)^{1/d_1}\big)\,c_2(t_2)}\longrightarrow f_{\rho_1}(x_1)f_{\rho_2^*}(x_2)\,.\,\,.
\end{equation*}
First of all, note that when $k=2$ we have $$\|C^{\circ 2}\,\|_{L^1((t_1D_1\times t_2D_2)^2)}=\int_{(t_1D_1\times t_2D_2)^2}C^2(x-y)\,dxdy=\frac{1}{2}{\rm Var}\left(\int_{(t_1D_1\times t_2D_2)} H_2(B_x)\,dx\right)$$ and the result is a consequence of \cref{prop:variance}, case (4). Moreover, by covariograms' based techniques (see e.g. \cite[p.750]{MN22}), since $C^2\ge0$ we have 
\begin{equation}
    \label{eq: finite L2 norm}
    \sup_{t_1,t_2} \left\|K_{t_1,t_2} \right\|^2_{L^2(D-D)} \asymp \sup_{t_1,t_2} \frac{\int_{(t_1D_1\times t_2D_2)^2}C^2(x-y)\,dxdy}{\left[t_1^{d_1}\,c_1\left(t_1c_2(t_2)^{1/d_1}\right)\right]^{2}\,\left[t_2^{d_2}\,c_2(t_2)\right]^{2}}<\infty\,,
\end{equation}
where the $\sup$ is finite because the sequence is convergent.

\noindent
Now, we prove the result in general. From now on, assume $k\ge3$. Fix $X^{(1)},X^{(2)}>0$ and define \begin{equation*}
    \mathcal F_{t_1,t_2}
:=
\Big\{x_1\in D_1-D_1:\ 
\|x_1\|\ge \frac{X^{(1)}}{t_1c_2(t_2)^{1/d_1}}\Big\},
\quad
\mathcal \mathcal D_{t_2}
:=
\Big\{x_2\in D_2-D_2:\ 
\|x_2\|\ge \frac{X^{(2)}}{t_2}\Big\}.
\end{equation*}
Choosing $X^{(1)},X^{(2)}>0$ according to Proposition \ref{prop:potter} and reasoning exactly as in \eqref{eq:since}, one can show that $K_{t_1,t_2}\mathbbm{1}_{(\mathcal F_{t_1,t_2}\times\mathcal D_{t_2})}$ is bounded by a locally $L^2$ dominating function. Thus, by \eqref{eq:locallyint cyclic} and dominated convergence we get \begin{equation*}
    \left\|\left(K_{t_1,t_2}\mathbbm 1_{(\mathcal F_{t_1,t_2}\times\mathcal D_{t_2})}\right)^{\circ k}\right\|_{L^1(D^k)} \longrightarrow\left\|f_{\rho_1}^{\circ k}\,\right\|_{L^1(D_1^k)}\left\|f_{\rho_2^*}^{\circ k}\,\right\|_{L^1(D_2^k)}\,.
\end{equation*}
It remains to control the contribute on $(\mathcal F_{t_1,t_2}\times\mathcal D_{t_2})^c:=(D-D)\setminus(\mathcal F_{t_1,t_2}\times\mathcal D_{t_2})$. Decompose \begin{equation*}
    K_{t_1,t_2}
=
K_{t_1,t_2}\mathbbm 1_{(\mathcal F_{t_1,t_2}\times\mathcal D_{t_2})}
+
K_{t_1,t_2}\mathbbm 1_{(\mathcal F_{t_1,t_2}\times\mathcal D_{t_2})^c}.
\end{equation*}
Then, by positivity of $C$ (recall \Cref{ass:KEY}) and multilinaerity of the cyclic product, we have \begin{equation*}
    \left\|K_{t_1,t_2}^{\circ k}\right\|_{L^1(D^k)}=\left\|\left(K_{t_1,t_2}\mathbbm 1_{(\mathcal F_{t_1,t_2}\times\mathcal D_{t_2})}\right)^{\circ k}\right\|_{L^1(D^k)}+r_{t_1,t_2}\,,
\end{equation*}
where \begin{equation*}
    r_{t_1,t_2}\le k\,\left\|\left(K_{t_1,t_2}\mathbbm 1_{(\mathcal F_{t_1,t_2}\times\mathcal D_{t_2})^c}\right)\circ\left(K_{t_1,t_2}\right)^{\circ (k-1)}\right\|_{L^1(D^k)} \,\,.
\end{equation*}
Moreover, by \eqref{eq:locallyint cyclic} we get \begin{equation*}
    r_t\lesssim \left\|K_{t_1,t_2}\mathbbm 1_{(\mathcal F_{t_1,t_2}\times\mathcal D_{t_2})^c}\right\|_{L^1(D-D)}\,\left\|K_{t_1,t_2}\right\|_{L^2(D-D)}^{k-1}
\end{equation*}
Therefore, thanks to \eqref{eq: finite L2 norm}, it only remains to prove that \begin{equation*}
    \left\|K_{t_1,t_2}\mathbbm 1_{(\mathcal F_{t_1,t_2}\times\mathcal D_{t_2})^c}\right\|_{L^1(D-D)}\rightarrow 0\,.
\end{equation*}
To do this, we use triangular inequality and split $(\mathcal F_{t_1,t_2}\times\mathcal D_{t_2})^c$ into the three disjoint regions \begin{equation*}
    (\mathcal F_{t_1,t_2})^c\times \mathcal D_{t_2},\qquad
\mathcal F_{t_1,t_2}\times(\mathcal \mathcal D_{t_2})^c,\qquad
(\mathcal F_{t_1,t_2})^c\times(\mathcal \mathcal D_{t_2})^c\,.
\end{equation*}
Choosing $X^{(1)},X^{(2)}>0$ according to Proposition \ref{prop:potter}, observing that $(\mathcal{F}_{t_1,t_2})^c\subseteq (\mathcal{F}_{t_1,0})^c$ since $c_1,c_2$ are non-decreasing and using \eqref{eq:nice lower} in \eqref{second rem}, we get:
\begin{align}  
    \label{first rem}
&K_{t_1,t_2}\mathbbm{1}_{(\mathcal F_{t_1,t_2})^c\times\mathcal \mathcal D_{t_2}}\lesssim \frac{f_{\rho_2+\delta}(x_2)\mathbbm{1}_{(\mathcal F_{t_1,t_2})^c\times\mathcal \mathcal D_{t_2}}}{c_1(t_1\,c_2(t_2)^{1/d_1})}\,\,;
\\
&\label{second rem}
K_{t_1,t_2}\mathbbm{1}_{\mathcal F_{t_1,t_2}\times \mathcal D_{t_2}^c}\,\,\lesssim \,\,\frac{f_{\rho_1+\delta}(x_1)\mathbbm{1}_{\mathcal F_{t_1,0}\times \mathcal D_{t_2}^c}}{c_2(t_2)}\frac{c_1(t_1)}{c_1(t_1 c_2(t_2)^{1/d_1})}\,\,;
\\ 
&\label{third rem}
K_{t_1,t_2}\mathbbm{1}_{(\mathcal F_{t_1,t_2})^c\times \mathcal D_{t_2}^c}\lesssim\frac{\mathbbm{1}_{(\mathcal F_{t_1,t_2})^c\times \mathcal D_{t_2}^c}}{c_2(t_2)c_1(t_1\,c_2(t_2)^{1/d_1})} \,\,.
\end{align}
Thus, we conclude using \cref{prop:potter}, \eqref{eq:regvar 0 e infty}, \Cref{ass:KEY} and the assumptions on $\rho_1,\rho_2$:
\begin{align}  
    \left\|K_{t_1,t_2}\mathbbm{1}_{(\mathcal F_{t_1,t_2})^c\times\mathcal \mathcal D_{t_2}}\right\|_{L^1(D-D)}&\lesssim \frac{1}{t_1^{d_1}c_2(t_2)\,c_1(t_1\,c_2(t_2)^{1/d_1})}\lesssim \frac{1}{(t_1c_2(t_2)^{1/d_1})^{d_1-\rho_1-\delta}} \nonumber \\
& \longrightarrow 0\,\,; \label{first rem normed}
\\
\label{second rem normed}
\left\|K_{t_1,t_2}\mathbbm{1}_{\mathcal F_{t_1,t_2}\times \mathcal D_{t_2}^c}\right\|_{L^1(D-D)}\,\,&\lesssim \,\,\frac{1}{t_2^{d_2}c_2(t_2)}\frac{c_1(t_1)}{c_1(t_1 c_2(t_2)^{1/d_1})}\lesssim \frac{1}{t_2^{d_2-\rho_2(1-\rho_1/d_1)-\delta}} \longrightarrow 0\,\,;
\\ 
\left\|K_{t_1,t_2}\mathbbm{1}_{(\mathcal F_{t_1,t_2})^c\times \mathcal D_{t_2}^c}\right\|_{L^1(D-D)}& \lesssim\frac{1}{t_1^{d_1}t_2^{d_2}c_2(t_2)c_1(t_1\,c_2(t_2)^{1/d_1})}\lesssim \frac{1}{t_1^{d_1-\rho_1-\delta}t_2^{d_2-\rho_2(1-\rho_1/d_1)-\delta}} \nonumber \\
& \longrightarrow 0\,\,.\label{third rem normed}
\end{align}
\end{proof}

\paragraph{Acknowledgments}
N.L. was partially supported by the ARC Discovery Grant DP220101680 (Australia). The research of L.M. has been supported by the PRIN Project 2022 GRAFIA funded by MUR, by the Excellence Department Project MatMod@TOV and by the INdAM group GNAMPA. I.N. gratefully acknowledges support from the Luxembourg National Research Fund (Grant No. O22/17372844/FraMStA). 

\bibliographystyle{abbrv}
\bibliography{pDOM}

\end{document}

%% file: macros.tex
\usepackage[colorinlistoftodos]{todonotes}
\usepackage{amsmath}
\usepackage{amsfonts}
\usepackage{amssymb}
\usepackage{amsthm}
\usepackage{tikz}
\usetikzlibrary{patterns, patterns.meta}
\usepackage{bbm}
\usepackage{mathtools,thmtools}
\usepackage{nccmath}
\usepackage{mathabx}
\usepackage{graphicx}
\usepackage{enumitem}
\usepackage{xcolor}
\usepackage{pifont}
\allowdisplaybreaks[4]

\usepackage{framed}
\numberwithin{equation}{section}

\usepackage{hyperref}
\hypersetup{
	colorlinks=true,     
	linkcolor=blue,         
	citecolor=red}

\numberwithin{equation}{section}

\usepackage{cleveref}

\DeclareMathOperator{\Var}{\mathbb{V}ar}

\DeclareMathOperator{\vol}{Vol}
\DeclareMathOperator{\Cov}{\mathbb{C}ov}

\DeclareMathOperator{\diam}{diam}

\newcommand{\eqlaw}{\overset{d}{=}}
\newcommand{\convlaw}{\overset{d}{\longrightarrow}}
\newcommand{\eqLtwo}{\overset{L^2}{=}}

\newcommand{\R}{\mathbb{R}}
\newcommand{\N}{\mathbb{N}}
\newcommand{\E}{\mathbb{E}}

\newcommand{\Z}{\mathbb{Z}}

\newcommand{\f}{\varphi}
\renewcommand{\phi}{\varphi}

\newcommand{\scp}[1]{\langle #1 \rangle}

\newcommand{\freecontr}[1]{\overset{r}{\frown}}

\theoremstyle{plain}
\newtheorem{theorem}{Theorem}[section]
\newtheorem*{theorem*}{Theorem}
\newtheorem{proposition}[theorem]{Proposition}
\newtheorem*{proposition*}{Proposition}

\newtheorem{corollary}[theorem]{Corollary}

\theoremstyle{definition}
\newtheorem{definition}{Definition}[section]
\newtheorem*{definition*}{Definition}
\newtheorem{remark}[definition]{Remark}
\newtheorem*{remark*}{Remark}
\newtheorem{example}[definition]{Example}
\newtheorem{assumption}{Assumption}[section]

\crefname{equation}{Equation}{Equations}
\crefname{multline}{Equation}{Equations}
\crefname{figure}{Figure}{Figures}
\crefname{question}{Question}{Question}
\crefname{section}{Section}{Sections}
\crefname{subsection}{Subsection}{Subsections}
\crefname{lemma}{Lemma}{Lemmas}
\crefname{proposition}{Proposition}{Propositions}
\crefname{theorem}{Theorem}{Theorems}
\crefname{innercustomthm}{Theorem}{Theorems}
\crefname{mainthm}{Theorem}{Theorems}
\crefname{corollary}{Corollary}{Corollaries}
\crefname{definition}{Definition}{Definitions}
\crefname{remark}{Remark}{Remarks}
\crefname{proposition}{Proposition}{Proposition}
\crefname{corollary}{Corollary}{Corollaries}
\crefname{example}{Example}{Examples}
\crefname{claim}{Claim}{Claim}
\crefname{conjecture}{Conjecture}{Conjecture}
\crefname{enumi}{}{}
\crefname{ass}{Assumption}{Assumptions}

